\RequirePackage{fix-cm}
\documentclass[smallextended]{svjour3}       
\smartqed  

\usepackage{amsmath,amssymb,subcaption,xspace,bbold,multirow,grffile,rotating,mathtools}
\usepackage{amsfonts}
\usepackage[ruled]{algorithm2e}
\RestyleAlgo{boxruled}
\usepackage{algorithmic,url}
\usepackage{color}
\usepackage[pointedenum]{paralist}
\usepackage{listings}
\usepackage{cases}
\usepackage{tikz}
\usetikzlibrary{matrix,decorations.pathreplacing}
\usepackage{graphicx}
\usepackage{subcaption}
\usepackage{grffile}
\usepackage[colorlinks=true]{hyperref}

\begin{document}
\SetKwComment{Comment}{$\triangleright$\ }{}
\newcommand{\mpi}{$\mbox{{\sf MPI-LIBLINEAR}}$\xspace}
\def\webspam{{\sf webspam}\xspace}
\def\news{{\sf news20}\xspace}
\def\eps{{\sf epsilon}\xspace}
\def\uu{{\sf url}\xspace}
\def\rcvt{{\sf rcv1t}\xspace}
\def\kddb{{\sf KDD2010-b}\xspace}
\def\indicator{{ \mathbb{1}}}
\def\R{{ \mathbf{R}}}
\def\Z{{ \mathbf{Z}}}
\def\N{{ \mathbf{N}}}
\def\POS{{\sf POS}\xspace}
\def\dep{{\sf DEP}\xspace}
\def\ADMM{{\sf ADMM-Struct}\xspace}
\def\perceptron{{\sf Distributed Perceptron}\xspace}
\def\bx{{\boldsymbol x}}
\def\by{{\boldsymbol y}}
\def\bz{{\boldsymbol z}}
\def\bb{{\boldsymbol b}}
\def\bp{{\boldsymbol p}}
\def\w{{\boldsymbol w}}
\def\bw{{\boldsymbol w}}
\def\bu{{\boldsymbol u}}
\def\bv{{\boldsymbol v}}
\def\bs{{\boldsymbol s}}
\def\be{{\boldsymbol e}}
\def\bd{{\boldsymbol d}}
\def\b1{{\boldsymbol 1}}
\def\bzero{{\boldsymbol 0}}
\def\bxi{\boldsymbol \xi}
\def\bpsi{\boldsymbol \psi}
\def\AL{{\boldsymbol{\alpha}}}
\def\DAL{{\boldsymbol{\Delta\alpha}}}
\def\HAL{{\boldsymbol{\hat{\alpha}}}}
\def\BL{\boldsymbol \beta}
\def\byi{{\boldsymbol{y}_i}}
\def\Bxi{\boldsymbol \xi}
\def\xset{{\mathcal{X}}}
\def\yset{{\mathcal{Y}}}
\newcommand{\dsvm}{$\mbox{{\sf DSVM-AVE}}$\xspace}
\newtheorem{assumption}{Assumption}
\newcommand{\disdca}{$\mbox{{\sf DisDCA}}$\xspace}
\newcommand{\bqo}{$\mbox{{\sf BDA}}$\xspace}
\newcommand{\blockapprox}{$\mbox{{\sf BDA}}$\xspace}
\newcommand{\lcommdir}{$\mbox{{\sf L-CommDir}}$\xspace}
\def\tron{{\sf TRON}\xspace}
\def\qedhere{}

\titlerunning{Distributed Block-diagonal Approximation for Dual ERM}
\title{{Distributed Block-diagonal Approximation Methods for
Regularized Empirical Risk Minimization}\thanks{Version of \today.}}
\author{Ching-pei Lee \and Kai-Wei Chang}
\institute{Ching-pei Lee\at
	Department of Mathematics and Institute for Mathematical Sciences,
	National University of Singapore\\ 
	\email{leechingpei@gmail.com}
	\and
	Kai-Wei Chang \at
	Department of Computer Science ,University of California Los
	Angeles\\
	\email{kw@kwchang.net}
}
\date{received: date / accepted: date}
\newcommand{\citet}[1]{\cite{#1}}
\newcommand{\citep}[1]{\cite{#1}}
\journalname{Machine Learning}
\maketitle
\begin{abstract}
	In recent years, there is a growing need to train machine learning
	models on a huge volume of data.  Therefore, designing efficient
	distributed optimization algorithms for empirical risk
	minimization (ERM) has become an active and challenging
	research topic.
	In this paper, we propose a flexible framework for distributed ERM
	training through solving the dual problem, which provides a unified
	description and comparison of existing methods.
	Our approach requires only approximate solutions of the
	sub-problems involved in the optimization process, and is
	versatile to be applied on many large-scale machine learning
	problems including classification, regression, and structured
	prediction.
	We show that our framework enjoys global linear convergence for a
	broad class of non-strongly-convex problems, and some specific
	choices of the sub-problems can even achieve much faster
	convergence than existing approaches by a refined analysis.
	This improved convergence rate is also reflected in the superior
	empirical performance of our method.
\end{abstract}


\section{Introduction}
\label{sec:intro}
With the rapid growth of data volume and model complexity, designing
scalable learning algorithms has become increasingly important.
Distributed optimization techniques, which distribute the
computational burden across multiple machines, have shown early
success on this path.
This type of approaches are particularly useful when the optimization
problem involves massive computation or when the dataset is stored
across multiple computational nodes.
However, the communication cost and the asynchronous nature of
distributed computation challenge the design of efficient optimization
algorithms in the distributed environment.

In this paper, we study distributed optimization algorithms for
training machine learning models that can be represented by the
regularized empirical risk minimization (ERM) framework.
Given a set of training data, $\{X_i\}_{i=1}^l, X_i \in \R^{n \times
c_i}, c_i\in \N$, where $l,n>0$ are the number of instances and the
dimension of the model respectively, regularized ERM models solve the
following optimization problem:
\begin{equation}
	\min_{\bw\in \R^n} \quad f^P(\bw) \coloneqq g(\bw) + \sum_{i = 1}^l
	\xi_i\left(X_i^T\bw\right).
	\label{eq:primal}
\end{equation}
In the literature, $g$ and $\xi_i$ are called the
regularization term and the loss term, respectively.
We assume that $f^P$ is a proper and closed convex function that can be
extended-valued and the solution set of \eqref{eq:primal} is nonempty.
Besides, we specifically focus on linear models, which have been shown
successful in dealing with large-scale data thank to their efficiency
and interpretability.\footnote{Linear models allow developers to
	interpret the value of each feature from the learned model
parameters.  }

The definition in problem  \eqref{eq:primal}
is general and covers a variety of learning problems, including binary
classification, multi-class classification, regression, and structured
prediction.
To unify different learning problems, we encode the true labels
(i.e., $y_i \in \mathcal{Y}_i$) in the loss term $\xi_i$ and the input
data $X_i$.
For some learning problems, the space of $X_i$ is spanned by a set of
variables whose size may vary for different $i$. Therefore, we
represent $X_i$ as an $n \times c_i$ matrix. For example, in the
part-of-speech tagging task, where each input sentence consists of a
sequence of words, $c_i$ represents the number of words in the $i$-th
sentence. We discuss in details the loss terms for
different learning problems in Section \ref{sec:applications}.
Regarding the regularization term, common choices include the
squared-$\ell_2$ norm, the $\ell_1$ norm, and the elastic net
that combines both \citep{HZ05a}.

In many applications, it is preferable to solve the dual
problem whose optimization might be easier.
The dual problem of \eqref{eq:primal} is
\begin{equation}
	\min_{\AL \in \Omega}\quad f(\AL) \coloneqq g^*(X \AL) + \sum_{i=1}^l
	\xi_i^*(-\AL_i),
	\label{eq:dual}
\end{equation}
where
\begin{equation*}
	X \coloneqq [X_1,\ldots, X_l],\quad
	\AL \coloneqq \left[
		\begin{array}{c}
			\AL_1\\\vdots\\\AL_l
		\end{array}
	\right],
\end{equation*}
$\AL_i \in \R^{c_i}$ is the dual variable vector associated with $X_i$,
for any function $h(\cdot)$, $h^*(\cdot)$ is the convex conjugate
of $h(\cdot)$:
\begin{equation*}
h^*(\bw) \coloneqq \max_{\bz \in \text{dom}(h)} \quad \bz^T  \bw - h(\bz),
\quad \forall \bw,
\end{equation*}
and as $g^*$ is finite everywhere in our setting (see
Assumption~\ref{assum:strong} and the description that follows), the
domain $\Omega$ is
\begin{equation*}
\Omega \coloneqq \prod_{i=1}^l
-\text{dom}(\xi_i^*) \subseteq \R^{\sum_{i=1}^l c_i}.
\end{equation*}

The goal of solving the dual problem \eqref{eq:dual} is still getting
a solution to the original primal problem \eqref{eq:primal}.
It can be shown easily by Slater's condition that when $f^P$ is
convex, strong duality between \eqref{eq:primal} and \eqref{eq:dual}
holds,
which means that any pair of primal and dual optimal solutions
$(\bw^*, \AL^*)$ satisfies
\begin{equation*}
f^P\left(\bw^*\right) = - f\left(\AL^*\right).
\end{equation*}

Despite that optimization methods for the dual ERM problem \eqref{eq:dual}
on a single machine have been widely studied (see, for example, the
survey paper \cite{GXY11c}), adapting them to a distributed
environment is not straightforward due to the following two reasons.
First, most existing
single-core algorithms for dual ERM are inherently sequential and
hence hard to parallelize.
Second, in a distributed environment, inter-machine communication is
usually the bottleneck for parallel optimizers and a careful design to
reduce the communication cost is essential.  For example, we may
prefer an algorithm with faster convergence even if it takes longer
computation time at each iteration as it induces fewer communication rounds and
consequentially reduces overall communication overhead.


In this paper, we propose a distributed learning framework for solving
\eqref{eq:dual}.
At each iteration, it minimizes a sub-problem consisting of the sum of
a second-order approximation of $g^*(X\AL)$ and the original
$\bxi^*(-\AL)$.
We study how to choose the approximation to let the proposed
approach enjoy not only fast theoretical and empirical convergence
rate but low communication overhead.
After solving the sub-problem, we conduct a line search that
requires only negligible computational cost and $O(1)$ communication
to ensure sufficient function value decrease.
With this line search procedure, our algorithm achieves faster
empirical performance compared with existing approaches.

By utilizing relaxed conditions, even if the subproblem is solved only
approximately, our method is able to achieve global
linear convergence for many problems whose dual objective is
not strongly convex, including support vectors machines
(SVM)~\citet{BB92a,VV95a} and structured support vector machines
(SSVM)~\citep{IT05a,BT04a}.
In other words, our algorithm takes only $O(\log(1/\epsilon))$
iterations, or equivalently, rounds of communication, to obtain an
$\epsilon$-accurate solution for \eqref{eq:dual}.\footnote{
Given any optimization problem
$\min_{x \in X} \, f(x)$
whose minimum is attainable and denoted by $f^*$,
we call $x \in X$ an $\epsilon$-accurate solution for
this problem if
$f(x) - f^* \leq \epsilon$.
}

Our analysis then shows that this result implies that obtaining
an $\epsilon$-accurate solution for the original ERM problem
\eqref{eq:primal} also takes only $O(\log(1 / \epsilon))$ iterations.
We further show that when the choice of the sub-problem properly
extracts information from the Hessian of $g^*(X\AL)$, the convergence
can be significantly accelerated to reduce the required number of
iterations and therefore the running time.
Besides, our flexible framework generalizes existing approaches and
hence facilitates the discussion of the differences between the
proposed distributed learning algorithms and the existing ones for
\eqref{eq:dual}.

Recently, \citet{ZS17a} proposed an accelerated
method for solving the dual ERM problem in a distributed setting.
Their techniques derived from \cite{SSS13a} is similar to the
Catalyst framework for convex optimization \citep{HL15a}.
In essence, at every iteration, their approach adds a term $\kappa \|\bw -
\bz\|_2^2/2$ to \eqref{eq:primal} and approximately solves the dual
problem of the modified primal problem by an existing
distributed optimization algorithm for \eqref{eq:dual}. The solution
is then used to generate $\bz$ for the next iteration.
Like the Catalyst framework that can be combined with any convex
optimization algorithm,
the acceleration technique in \cite{ZS17a} can also be incorporated
with our distributed learning algorithm.
Specifically, we can apply our proposed algorithm to solve the
modified dual problem in the procedure mentioned above.
Therefore, to simplify the discussion, we focus on comparing methods
that solve the original optimization problem \eqref{eq:dual},
and acceleration approach discussed in \cite{ZS17a} and other
studies not designed for distributed learning (e.g.,
\cite{HL15a,SSS13a}) are not in the scope of this study.

Different from approaches that consider the theoretical communication
efficiency only,
our goal is to design a practical distributed training
algorithm for regularized ERM with strong empirical performance in
terms of the overall running time.
Therefore, we propose an algorithm that is both computation and
communication efficient by designing a second-order method, in which
the approximated Hessian can be computed without lengthy rounds of
communication.
We show that this approach is also extremely communication-efficient
in theory by taking the approximated Hessian as a preconditioner
that can significantly improve the condition number of the problem.

Special cases of the proposed framework were published earlier as
conference and workshop papers \citep{CPL15a,CPL15b}.
In the journal version, we unify the results and extend the previous
work to a general setting that covers a much broader class of
problems, and provide thorough theoretical and empirical analyses.
We show that either when the sub-problem is solved exactly or
approximately at every iteration, our approach enjoys fast linear
convergence, and the convergence rate behaves benignly with respect to
the inexactness.
We also provide a novel analysis showing why the selected sub-problem can
greatly improve the convergence speed when compared to existing general analyses.

\paragraph{Contributions}
We propose a general framework for optimizing the dual ERM problem
\eqref{eq:dual} when the data are stored on multiple machines. Our
contributions are summarized in the following.
\begin{enumerate}
\item Our framework is flexible, allowing different choices of
	sub-problem formulations, sub-problem solvers, and line search
	approaches. Furthermore, approximate sub-problem solutions can be
	used.  As a result, many existing methods can be viewed as special
	cases of our framework.
\item We provide detailed theoretical analysis, showing that our
	framework converges linearly on a class of problems broader than
	the strongly convex ones, even when the sub-problem is solved only
	approximately.  Our analysis not only shows fast convergence of the
	proposed algorithm, it also provides sharper convergence
	guarantees for existing methods that can fit into our framework.
\item We further give an analysis through change of norm to show that
	under specific sub-problem choices, our algorithm can achieve much
	faster convergence.
	Our analysis gets around the dependency on the condition number
	defined by the Euclidean norm and can therefore obtain faster
	rates than existing approaches.
\item The proposed approach is also empirically communication- and
	computation-efficient. Our empirical study shows that it
	outperforms existing methods on real-world large-scale datasets.
\end{enumerate}

\paragraph{Notations}
We use the following notations.
\begin{equation*}
	\bxi(X^T\bw) \coloneqq \sum_{i=1}^l \xi_i(X_i^T\bw), \quad
	\bxi^*(-\AL) \coloneqq \sum_{i=1}^l \xi^*_i(-\AL_i), \quad
	G^*(\AL) \coloneqq g^*(X\AL).
\end{equation*}
For any positive integer $m$, any vector $\bv \in \R^m$, and any
	$I \subseteq \{1,\dots,m\}$, $\bv_I$ denotes the sub-vector in
	$\R^{|I|}$ that contains the coordinates of $\bv$ indexed by $I$.
	We use  $\|\cdot\|$ to denote the Euclidean norm, and
	when given a symmetric positive semidefinite matrix $A$, we denote the
	seminorm induced by it as
	\begin{equation*}
		\|x\|_A \coloneqq \sqrt{x^T A x}.
	\end{equation*}

\paragraph{Assumptions}
We consider the following setting in this paper.
First, we assume the training instances are
	distributed across $K$ machines, where the instances on machine
	$k$ are $\{X_i \}_{i \in J_k}$.  In our setting, $J_k$ are
	disjoint index sets such that
\begin{equation*}
	\bigcup_{k = 1}^K J_k = \{1,\ldots,l\},\quad
	J_i \cap J_k = \phi, \quad \forall i\neq k.
\end{equation*}
Without loss of generality, we assume that there is a sequence of
non-decreasing non-negative integers
\begin{equation*}
	0 = j_0 \leq j_1\leq\dotsc \leq j_K = l
\end{equation*}
such that
\begin{equation*}
	J_k = \left\{j_{k-1}+1,\dotsc,j_k\right\},\quad
	k=1,\dotsc, K.
\end{equation*}
We do not make any further assumption on how those instances are distributed
across machines. That is, the data distributions on different machines
can be different.
Second, the problem is assumed to have the following properties.
\begin{assumption}
\label{assum:strong}
The loss function $\xi_i$ is convex for all $i$ and
there exists $\sigma > 0$ such that the regularizer $g$ in
the primal problem \eqref{eq:primal} is $\sigma$-strongly convex.
Namely,
\begin{align}
\nonumber
&g(\alpha \bw_1 + (1 - \alpha) \bw_2) \leq \alpha g(\bw_1) + (1 -
\alpha) g(\bw_2) - \frac{\sigma\alpha(1 -\alpha)}{2} \|\bw_1 -
\bw_2\|^2, \\ 
&\forall \bw_1, \bw_2 \in \R^n,\quad \forall \alpha \in [0,1].
\label{eq:strongcvx}
\end{align}
Moreover, the convex function $f^P(\bw)$ is proper and closed, and
\eqref{eq:primal} has a non-empty solution set.
\end{assumption}
Since $g$ is $\sigma$-strongly convex, $g^*$ is differentiable and has
$(1/\sigma)$-Lipschitz continuous gradient~\cite[Part E, Theorem
4.2.2]{HU01a}.
Therefore,  $G^*$ has $(\|X^T X\|/\sigma)$-Lipschitz continuous
gradient.
This also indicates that even if $g$ is extended-valued, $g^*$ is
still finite everywhere, hence the only constraint on the feasible
region is from the domain of $\bxi^*$, namely $\AL \in \Omega$.

\paragraph{Organization}
The paper is organized as follows.
We give an overview of the proposed framework in Section \ref{sec:method}.
Implementation details and convergence analysis are respectively
discussed in Sections~\ref{sec:implement} and \ref{sec:analysis}.
We summarize related studies for distributed ERM optimization in
Section \ref{sec:related} and discuss applications of the proposed
approach in Section \ref{sec:applications}.
We then demonstrate the empirical performance of the proposed
algorithms in Section \ref{sec:exp} and discuss possible extensions and
limitations of this work in Section \ref{sec:discuss}.
The conclusions and final remarks are in Section \ref{sec:conclusion}.

The code for reproducing the experimental results in this paper is
available at \url{http://github.com/leepei/blockERM}.

\section{A Block-diagonal Approximation Framework}
\label{sec:method}
As $g^*$ is differentiable from Assumption~\ref{assum:strong},
the KKT optimality conditions imply that for any pair of
primal and dual optimal solutions $(\bw^*, \AL^*)$,
\begin{equation}
\bw^* = \nabla g^*(X\AL^*).
\label{eq:kkt}
\end{equation}
Although \eqref{eq:kkt} holds only at the optima,
we take the same formulation to define $\bw(\AL)$ as the primal
iterate associated with any $\AL$ for the dual problem
\eqref{eq:dual}:
\begin{equation}
	\label{eq:w}
\bw(\AL) \coloneqq \nabla g^*(X\AL).
\end{equation}

Our framework is an iterative descent method for solving
\eqref{eq:dual}.
Starting with an arbitrary feasible $\AL^0$, it generates a sequence
of feasible iterates $\{\AL^1,\AL^2,\dotsc\} \subset \Omega$
with the property that $f(\AL^i) \leq f(\AL^j)$ if $i > j$.
Each iterate is updated by a direction $\DAL^t$ and a step
size $\eta_t \geq 0$.
\begin{equation}
	\AL^{t+1} = \AL^t + \eta_t \DAL^t, \quad \forall t \geq 0.
	\label{eq:update}
\end{equation}
The term $\bxi^*$ in \eqref{eq:dual} is separable in $\AL$ and hence
can be optimized directly in a coordinate-wise manner.
However, the term $G^*$ is often complex and difficult to optimize.
Therefore, we approximate it using a quadratic surrogate based on the
fact that $G^*$ is Lipschitz-continuously differentiable.

Putting them together, given the current iterate $\AL^t$, we solve
\begin{equation}
    \begin{split}
    \DAL^t &\approx \arg\min_{\DAL} \, Q_{B_t}^{\AL^t}(\DAL), \\
	Q_{B_t}^{\AL^t}\left( \DAL \right)& \coloneqq
	\nabla G^*(\AL^t)^T
	\DAL + \frac12 (\DAL)^T B_t \DAL + \bxi^*(-\AL^t - \DAL)
	\label{eq:quadratic}
    \end{split}
\end{equation}
to obtain the update direction $\DAL^t$,
where $B_t$ for each $t$ is some symmetric matrix selected to approximate
$\nabla^2 G^*(\AL^t)$ (note that since $\nabla G^*$ is Lipschitz
continuous, $G^*$ is twice-differentiable almost everywhere so we at
least have the generalized Hessian),
and there is a wide range of choices for it, depending on the
scenario.
Note that it is usually hard to solve \eqref{eq:quadratic} to
optimality unless $B_t$ is diagonal. Therefore, we consider only
attaining approximate solutions for \eqref{eq:quadratic}.
We will discuss the selection of $B_t$ in Section \ref{sec:implement}.
The general analysis in Section \ref{sec:analysis} shows that as long as
the objective of \eqref{eq:quadratic} is strongly convex enough (in
the sense that the strong convexity parameter is large enough), even
if $B_t$ is indefinite and \eqref{eq:quadratic} is solved only
roughly, $\DAL^t$ will be a descent direction.
On the other hand, when $B_t$ approximates $\nabla^2 G^*(\AL^t)$ well
as in our choice,
the analysis in Section~\ref{subsec:improved} shows that the
convergence speed can be highly improved, making the algorithm more
communication-efficient.

Regarding the step size $\eta_t$, we investigate two line search strategies.
The first is the exact line search strategy, in which we minimize the
objective function along the update direction:
\begin{equation}
	\eta_t = \arg\min_{\eta\in \R}\quad f(\AL^t + \eta \DAL^t).
    \label{eq:exact_linesearch}
\end{equation}
However, this approach is not practical unless
\eqref{eq:exact_linesearch} can be solved easily.
Therefore, in general, we apply a backtracking line search strategy
using a modified Armijo rule suggested by \citet{PT07a}.
Given $\beta,\tau \in (0,1)$,
our procedure finds the smallest integer $i \geq 0$ such that
$\eta = \beta^i$ satisfies
\begin{equation}
	\label{eq:armijo}
	f(\AL^t + \eta \DAL^t) \leq f(\AL^t) + \eta \tau \Delta_t,
\end{equation}
where
\begin{equation}
	\Delta_t \coloneqq \nabla G^*(\AL^t)^T \DAL^t + \bxi^*(-\AL^t -
	\DAL^t) - \bxi^*(-\AL^t),
\label{eq:delta}
\end{equation}
and takes $\eta_t = \eta$.
Notice that as we are approximating the Hessian, similar to Newton and
quasi-Newton methods, our backtracking always starts from trying the
unit step size $\eta = 1$.

\section{Distributed Implementation for Dual ERM}
\label{sec:implement}
In this section, we provide technical details on how to apply the algorithm framework discussed in Section
\ref{sec:method} in a distributed environment.
In particular, we will discuss the choice of $B_t$ in
\eqref{eq:quadratic} such that the communication overhead can be reduced.
We will also propose a trick to make line search efficient.

For the ease of algorithm description, we denote the $i$-th column of $X$
by $\bx_i$, and the corresponding element of $\AL$ by $\alpha_i$.
We also denote the number of columns of $X$, which is equivalent to
the dimension of $\AL$, by
\begin{equation*}
	N \coloneqq \sum_{i=1}^l c_i.
\end{equation*}
The index sets corresponding to the columns of the instances in
$J_{k}$ are denoted by $\tilde{J}_k\subseteq \{1,\ldots,N\},
k=1,\ldots,K$.
We define
\begin{equation}
	\pi(i) = k, \quad\text{ if }i \in \tilde{J}_k.
\label{eq:pi}
\end{equation}

\subsection{Update Direction}
\label{subsec:direction}
In the following, we discuss how to select $B_t$ such that the
objective of \eqref{eq:quadratic} is 1) strongly convex, 
2) easy to optimize with low communication cost, and 3) a good approximation of \eqref{eq:dual}.

In our assumption, the $k$-th machine stores and handles only $X_i$ and the
corresponding $\AL_i$ for $i \in J_k$.
In order to reduce the communication cost, we need to pick
$B_t$ in a way such that \eqref{eq:quadratic} can be decomposed into
independent sub-problems, of which each involves only data points stored on
the same machine.
In such a way, each sub-problem can be solved locally on one node without
any inter-machine communication.
Motivated by this, $B_t$ should be block-diagonal (up to
permutations of the instance indices) such that
\begin{equation}
	\left( B_t \right)_{i,j} = 0,\text{ if }\pi(i) \neq \pi(j),
	\label{eq:blockdiagonal}
\end{equation}
where $\pi$ is defined in \eqref{eq:pi}.

The ideal choice for $B_t$ is to set it
to be the Hessian matrix $H_{\AL^t}$ of $G^*(\AL^t)$:
\begin{equation*}
	H_{\AL^t} \coloneqq \nabla^2 G^*(\AL^t)
	= X^T \nabla^2 g^*\left(X\AL^t\right) X.
\end{equation*}
This choice leads to the proximal Newton methods that enjoys rapid
convergence in both theory and practice.
However, the Hessian matrix is usually dense and does not satisfy the
condition \eqref{eq:blockdiagonal}, incurring significant communication
cost in the distributed scenario we consider here.

Therefore, we consider a block-diagonal approximation
$\tilde{H}_{\AL^t}$ instead.
\begin{equation}
	\left( \tilde{H}_{\AL^t} \right)_{i,j} = \begin{cases}
		\left( H_{\AL^t}\right)_{i,j} & \text{ if }
	\pi(i) = \pi(j),\\
	0 & \text{ otherwise.  }
	\end{cases}
	\label{eq:blockHg}
\end{equation}
Note that since
\begin{equation*}
	\nabla^2_{i,j} G^*\left(X\AL^t\right) = \bx_i^T \nabla^2
g^*\left(X\AL^t\right) \bx_j,
\end{equation*}
if each machine maintains the whole vector of $X\AL^t$,
entries of \eqref{eq:blockHg} can be decomposed into parts such that
each one is constructed using only data points stored on one machine.
Thus, the sub-problems can be solved separately on different machines
without communication.
The Hessian matrix may be only positive semi-definite.
In this case, when $\bxi^*(-\AL)$ is not strongly convex, neither is
problem \eqref{eq:quadratic}, and the sub-problem can therefore be
ill-conditioned.
To remedy this issue, we add a damping term to $B_t$ to ensure strong
convexity of problem \eqref{eq:quadratic} when needed.

To summarize, our choice for $B_t$ in distributed environments can be
represented by the following formulation.
\begin{equation}
	B_t = a_1^t \tilde{H}_{\AL_t} + a_2^t I,\quad \text{ for
	some } a_1^t, a_2^t \geq 0.
	\label{eq:Ht}
\end{equation}
The values of $a_1^t$ and $a_2^t$ depend on the problem structure and
the applications.
In most cases, we set $a_2^t = 0$, especially when it is
known that either $\bxi^*(-\AL)$ is strongly convex, or
$\tilde{H}_{\AL_t}$ is positive definite.
For $a_1^t$, practical results
\citep{DP11a,TY13b} suggest that $a_1^t \in [1,K]$ leads to
good empirical performance,
while we prefer $a_1^t \equiv 1$ as it is a closer approximation to
the Hessian.

In solving \eqref{eq:quadratic} with our choice \eqref{eq:Ht} and
$a_1^t \neq 0$,
each machine needs the information of
$X\AL^t$ to calculate both $(B_t)_{\tilde J_k,\tilde J_k}$ and
\begin{equation}
	\nabla_{\tilde J_k} G^*\left( \AL^t \right) = X_{:,\tilde J_k}^T
\nabla g^*\left( X\AL^t \right).
\label{eq:gradG}
\end{equation}
Therefore, after updating $\AL^t$, we need to synchronize the
information
\begin{equation*}
	\bv^t \coloneqq X\AL^t = \sum_{k=1}^K \sum_{j \in J_k} X_j
	\AL^t_j
\end{equation*}
through one round of inter-machine communication.
Synchronizing this $n$-dimensional vector across machines is more effective than
transmitting either the Hessian or the whole $X$ together with
$\AL^t$.
However, we also need the update direction for line search; therefore,
instead of $\bv^{t+1}$, we synchronize
\begin{equation}
	\Delta \bv^t \coloneqq X \DAL^t = \sum_{k=1}^K \sum_{j \in J_k}
	X_j \Delta \AL^t_j
	\label{eq:deltaw}
\end{equation}
over machines  and then update $\bv^{t+1}$ locally on all machines by
\begin{equation*}
	\bv^{t+1} =\bv^t + \eta_t \Delta \bv^t
\end{equation*}
after the step size $\eta_t$ is determined.
Details of the communication overhead will be discussed in
Sections~\ref{subsec:cost} and \ref{sec:analysis}.

\subsection{Line Search}
\label{subsec:linesearch}
After the update direction $\DAL^t$ is decided by solving
\eqref{eq:quadratic} (approximately),
we need to conduct line search to find a step size satisfying
condition \eqref{eq:armijo} to ensure sufficient function value
decrease.
On the right-hand side of \eqref{eq:armijo}, the first term is
available from the previous iteration; therefore, we only need to
evaluate \eqref{eq:delta}.
From \eqref{eq:gradG}, this can be calculated by
\begin{equation}
	\Delta_t
= \nabla g^* \left( \bv^t \right)^T\Delta\bv^t +
\left(\bxi^* \left(- \AL^t - \DAL^t \right) - \bxi^*\left(-\AL^t
\right) \right).
\label{eq:deltasum}
\end{equation}
We require only $O(1)$ communication overhead to evaluate the
$\bxi^*$ functions in Eq. \eqref{eq:deltasum},
and no additional computation is needed because the information of
$\bxi^*(-\AL^t - \DAL^t)$ is maintained when solving \eqref{eq:quadratic}.
Furthermore, because each machine has
full information of $\Delta \bv^t$, $\bv^t$, and hence $g^*(\bv^t)$,
the first term in \eqref{eq:deltasum} can be calculated  in a
distributed manner as well to reduce the computational cost per machine.
Thus, we can combine the local partial sums of all terms in
\eqref{eq:deltasum} as a scalar value and synchronize it across machines.
One can also see from this calculation that synchronizing $\Delta
\bv^t$ is inevitable for computing the required values efficiently.

For the left-hand side of \eqref{eq:armijo}, the calculation of the
$\xi_i^*$ terms is distributed by nature as discussed above.
If $g^*(\bv)$ is separable, its computation can also
be parallelized.
Furthermore, in some special cases,  we are able to evaluate
$g^*(\bv + \eta \Delta \bv)$ using a closed-form formulation cheaply.
For example, when
\begin{equation*}
	g^*(\bv) = \frac{1}{2}\left\|\bv\right\|^2,
\end{equation*}
we have
\begin{equation}
	g^*\left(\bv + \eta \Delta \bv\right) = \frac{1}{2}\left(
	\left\|\bv\right\|^2 + \eta^2 \left\| \Delta\bv\right\|^2 + 2
	\eta \bv^T \Delta\bv \right).
	\label{eq:g}
\end{equation}
In this case, we can precompute $\|\Delta \bv\|^2$ and $\bv^T
\Delta\bv$, then the calculation of \eqref{eq:g} with
different $\eta$ requires only $O(1)$ computation without any
communication.
For the general case, though the computation might not be this low,
by maintaining both $\bv$ and $\Delta \bv$,
the calculation of $g(\bv + \eta \Delta \bv)$ requires no additional
communication and at most $O(n)$ computation locally, and this
cost is negligible as other parts of the algorithm incur
more expensive computation.
The line search procedure is summarized in Algorithm
\ref{alg:linesearch}.

The exact line search strategy is possible only when
\begin{equation*}
\frac{\partial
f(\AL+ \eta \DAL)}{ \partial \eta} = 0
\end{equation*}
has an analytic solution.
For example, when $f$ is quadratic,
we can compute
\begin{equation}
\frac{\partial
f(\AL+ \eta \DAL)}{ \partial \eta} = 0 \quad \Rightarrow \quad
\eta = \frac{-\nabla f(\AL)^T \DAL}{\DAL^T \nabla^2 f(\AL) \DAL},
\label{eq:quadlinesearch}
\end{equation}
and then project $\eta$ back to the interval
$\{\eta \mid \AL+\eta \DAL \in \Omega\}$.

\begin{algorithm}[t]
	\DontPrintSemicolon
	\caption{Distributed backtracking line search}
	\label{alg:linesearch}
\KwIn{
	$\AL, \DAL \in \R^N$,
	$\beta, \tau \in (0,1)$,
	$f(\AL) \in \R$,
	$\bv = X \AL, \Delta\bv = X \Delta \AL$}
	Form a partition $\{\hat{J}_k\}_{k=1}^K$ of $\{1,\dotsc,n\}$\;
	Calculate $\Delta_t$ in parallel: \Comment*[f]{$O(1)$
	communication}
\begin{equation*}
	\Delta_t = \sum_{k=1}^K\left( \nabla_{\hat{J}_k} g^* \left( \bv^t
	\right)^T
	\Delta\bv^t_{\hat{J}_k} + \sum_{j \in J_k}
	\xi_j^*\left(-\AL_j + \DAL_j\right) -
	\xi_j^*\left(-\AL_j\right)\right).
\end{equation*}\;
	$\eta \leftarrow 1$\;
Calculate $f(\AL + \eta \DAL)$ using $\bv$ and $\eta \Delta \bv$
\Comment*[r]{$O(1)$ communication}
	\While{$
	f(\AL+\eta \DAL) > f(\AL)  + \eta \tau \Delta_t$}
	{
		$\eta \leftarrow \eta\beta$\;
		Calculate $f(\AL + \eta \DAL)$ using $\bv$ and $\eta \Delta
		\bv$ \Comment*[r]{$O(1)$ communication}
	}
	\KwOut{$\eta$, $f(\AL + \eta \DAL)$}
\end{algorithm}

\subsection{Sub-Problem Solver on Each Machine}
If $B_t$ satisfies \eqref{eq:blockdiagonal},
\eqref{eq:quadratic} can be decomposed into $K$ independent
sub-problems:
\begin{equation}
	\min_{\DAL_{J_k}}\, \nabla_{J_k} G^*(\AL^t)^T
	\DAL_{J_k} + \frac12 \DAL_{J_k}^T (B_t)_{J_k,
	J_k} \DAL_{J_k} +
	\sum_{i \in J_k} \xi^*_i(-\AL^t_i - \DAL_i).
	\label{eq:subprob}
\end{equation}
Since all the information needed for solving \eqref{eq:subprob} is
available on machine $k$, the sub-problems can be solved without any
inter-machine communication.

Our framework does not pose any limitation on the solver for \eqref{eq:quadratic}.
For example, \eqref{eq:quadratic} can be solved by (block) coordinate
descent, (accelerated) proximal methods, just to name a few.
In our experiment, we use a random-permutation cyclic coordinate
descent method for the dual ERM problem~\citep{CJH08a,HFY10a,ChangYi13} as our
local solver. This method has been proven to be efficient in the
single-core setting empirically; theoretically, it is guaranteed to converge
globally linearly \cite{PWW13a} and can outperform other variants of
coordinate descent on some cases \citep{CPL18c,SJW17a}.
Other options can be adopted for specific problems or
datasets under discretion, but such discussion is beyond the scope of
this work.

\subsection{Output the Best Primal Solution}
\label{subsec:practical}
The proposed algorithm is a descent method for the dual problem
\eqref{eq:dual}.
In other words, it guarantees that $f(\AL^{t_1}) < f(\AL^{t_2})$ for
$t_1 > t_2$.
However, there is no guarantee that the corresponding primal solution
$\bw$ calculated by \eqref{eq:w} decreases monotonically as
well.\footnote{We will show in Section \ref{sec:analysis} that the
primal objective converges R-linearly, but there is no guarantee on
monotonic decrease.}
This is a common issue for all dual methods. To deal with it, we keep
track of the primal objectives of all iterates, and when the algorithm
is terminated, we report the model with the lowest primal objective.
This is known as the pocket approach in the literature of
Perceptron~\citep{SG90a}.

\subsection{Stopping Condition}
It is impractical to solve problem \eqref{eq:dual} exactly, as a model
that is reasonably close to the optimum can achieve similar or
even identical accuracy performance compared to the optimum.
In practice, one can design the stopping condition for the training
process by using the norm of the update direction, the size of
$\Delta_t$, or the decrement of the objective function value.
We consider the following practical stopping criterion:
\begin{equation*}
	f(\AL^t) + f^P(\bw(\AL^t)) \leq \epsilon \left( f \left( \AL^0 \right) +
	f^P \left( \bw \left( \AL^0 \right) \right) \right),
\end{equation*}
where $\epsilon \geq 0$ is a user-specified parameter.
This stopping condition directly reflects the model quality and is
easy to verify as the primal and dual objectives are computed at every iteration.

The overall distributed procedure for optimizing \eqref{eq:dual}
discussed in this section is described in Algorithm \ref{alg:blockcd}.

\begin{algorithm}[t]
	\DontPrintSemicolon
	\caption{Distributed block-diagonal approximation method for the
		dual ERM problem \eqref{eq:dual}.}
	\label{alg:blockcd}
	\KwIn{A feasible $\AL^0$ for
		\eqref{eq:dual}, $\epsilon \geq 0$}
		$\bar{f} \leftarrow \infty,$
		$\bar{w} \leftarrow \bzero$\;
		Compute $\bv^0 = \sum_{k=1}^K \sum_{j \in J_k} X_j
		\AL^0_j$  and $\bxi^*(-\AL^0)$
	\Comment*[r]{$O(n)$ communication}
	Compute $f(\AL^0)$ by $\bv^0$ and $\bxi^*(-\AL^0)$\;
	\For{$t = 0,1,2,\ldots$}{
		Compute $f^P(\bw(\AL^t))$ by \eqref{eq:w}
	\Comment*[r]{$O(1)$ communication}
		\If{$f^P(\bw(\AL^t)) < \bar{f}$}
		{$\bar{f} \leftarrow f^P(\bw(\AL^t))$, $\bar{w} \leftarrow
	\bw(\AL^t)$}
	\If{$f(\AL^t) + f^P(\bw(\AL^t)) \leq \epsilon (f(\AL^0) +
	f^P(\bw(\AL^0)))$}{Output $\bar{w}$ and terminate}
		Decide $a_1^t, a_2^t \geq 0$ but not both $0$\;
		Each machine obtains $\DAL_{J_k}^t$ by approximately solving
		\eqref{eq:subprob} independently and in parallel using the
		local data, with $B$ decided by \eqref{eq:Ht}\;
		Communicate $\Delta \bv^t = \sum_{k=1}^K \left( \sum_{j\in
			J_k} X_j \Delta \AL^t_j\right)$ \Comment*[r]{$O(n)$ communication}
		\begin{itemize}
			\item Variant I:
		Conduct line search through Algorithm \ref{alg:linesearch} to
		obtain $\eta_t$\;
		\item Variant II:
			$\eta_t \leftarrow \arg\min_\eta f(\AL^t + \eta \DAL^t)$\;
	\end{itemize}
	Each machine conducts in parallel:
	$\AL_{J_k}^{t+1} \leftarrow \AL_{J_k}^t + \eta_t \DAL^t_{J_k}$,
	$\bv^{t+1} \leftarrow \bv^t + \eta_t \Delta \bv^t$\;
	}
\end{algorithm}

\subsection{Cost per Iteration}
\label{subsec:cost}
In the following, we analyze the time complexity of each component in
the optimization process and summarize the cost per iteration of the
proposed algorithm.
For the ease of analysis, we assume that the number of columns of $X$
on each machine is $O(N / K)$, and the corresponding non-zero entries
on each machine is $O(\#\text{nnz} / K)$, where $\#\text{nnz}$ is the
number of non-zero elements in $X$.
We consider general $\xi^*$ and $g^*$ and
assume that the evaluations for $g(\bw)$, $g^*(\bv)$, and $\nabla
g^*(\bv)$ all cost $O(n)$.\footnote{We do not consider special cases
such as $g(\bw) = \|\bw\|^2/2$. In those cases, further acceleration
can be derived depending on the specific function structure.}
The part of $\nabla^2 g^*(\bv)$ is assumed to cost at most $O(n)$
(both for forming it and for its product with another vector), for
otherwise we can simply replace it with a diagonal matrix as an approximation.
In practice, performing exact line search is impractical unless the
problem structure allows.
Therefore, in the following, we only analyze the backtracking line
search strategy.
We also assume without loss of generality that the cost for evaluating
one $\xi^*_i$ is proportional to the dimension of the domain, namely
$O(c_i)$, so the evaluation of $\bxi^*$ costs $O(N/K)$ on each
machine.

We first check the cost for forming the problem \eqref{eq:quadratic}.
Note that we do not explicitly compute the values of $B_t$ and $\nabla
G^*(\AL^t)$.
Instead,
we compute only $\nabla G(\AL^t)^T \DAL^t$, through $\nabla
g^*(\bv^t)^T \Delta \bv^t$,
and the part $(\DAL)^T B_t \DAL$ under the choice \eqref{eq:Ht}
is obtained through $\|\DAL\|^2$ and $(\Delta \bv^t)^T \nabla^2
g^*(\bv^t)\Delta \bv^t$.
Therefore, for the linear term,
we need to compute only $\nabla g^*(\bv^t)$,
which costs $O(n)$ under our assumption given that $\bv^t$ is already
available on all the machines.
For the quadratic term,
it takes the same effort of $O(n)$ to get $\nabla^2 g^*(\bv^t)$.
Thus, forming the problem \eqref{eq:quadratic} costs $O(n)$ in
computation and no communication is involved.
Note that when considering the local sub-problems \eqref{eq:subprob},
the cost remains $O(n)$ as we just replace $\Delta \bv$ with the local
part $X_{:,J_k} \AL_{J_k}$, which is still a vector of dimension $n$.

Next, the cost of solving \eqref{eq:subprob} approximately by
passing through the data for a constant number of iterations $T$ is
$O(T \#\text{nnz} / K)$, as noted in most state-of-the-art single-core
optimization methods for the dual ERM (e.g., \cite{CJH08a,HFY10a}).
This part involves no communication between machines as well.

For the line search, as discussed in Section
\ref{subsec:linesearch},
we first need to make $\Delta \bv^{t}$ available on all machines.
The computational complexity for calculating $\Delta\bv^t$ through \eqref{eq:deltaw} is
$O(\#\text{nnz} / K)$,
and since the vector is of length $n$, it takes $O(n)$
communication cost to gather information from all machines.
After $\Delta \bv^t$ is available on all machines
and $\nabla g^*(\bv^t)$ is obtained,
we can calculate the first term of \eqref{eq:deltasum}.
This step costs $O(n/K)$ and $O(1)$ in computation and communication,
respectively.
The term related to $\bxi^*$ is a sum over $N$ individual functions
and therefore costs $O(N/K)$.
Thereafter, summing them up requires a $O(1)$ communication that can be
combined with the communication for obtaining $\nabla g^*(\bv^t)^T \Delta
\bv^t$.
Given $\bv^t$ and $\Delta \bv^t$, for each evaluation of $f$ under
different $\eta$,
it takes $O(n)$ to compute $\bv^t + \eta \Delta \bv^t$ and evaluate
the corresponding $g^*$.
For the part of $\bxi^*$, it costs $O(N/K)$ and $O(1)$ in computation and communication as it is
a sum over $N$ individual functions.
In total, each backtracking line search iteration costs $O(n + N/K)$
computation and $O(1)$ communication.

Finally, from \eqref{eq:w}, the vector $\bw(\AL^t)$ is the
same as the gradient vector we need in \eqref{eq:quadratic}, so there
is no additional cost to obtain the primal iterate,
and evaluating the primal objective costs $O(n)$ for $g(\bw(\AL^t))$
and $O(\#\text{nnz}/K)$ for $X^T \bw(\AL^t)$ in computation.
Thus the cost of the primal objective computation is $O(\#\text{nnz}/K
+ n)$.
It also takes $O(1)$ communication to gather the summation of
$\xi_i$ over the machines.

By assuming that each row and each column of $X$ has at least one
non-zero entry (for otherwise we can simply remove that row or
column),
we have $n + N = O(\#\text{nnz})$.
In summary, each iteration of Algorithm \ref{alg:blockcd} costs
\begin{equation*}
	O\left(\frac{\#\text{nnz}}{K} + n + \left(\frac{N}{K} +
	n\right) \times \#\text{(line search)}\right)
\end{equation*}
in computation and
\begin{equation*}
	O\left(n + \#\text{(line search)}\right)
\end{equation*}
in communication.
Later, we will show in Section \ref{sec:analysis} that the number of line search iterations is
upper-bounded by a constant. Therefore, 
the overall cost per iteration is $O(\#\text{nnz} /K + n)$
in computation and $O(n)$ in communication.

\section{Analysis}
\label{sec:analysis}
Our analysis consists of four parts.
The first three parts consider a general scenario and provide worst-case convergence guarantees of the proposed algorithm.
The last part considers a special choice of  $B_t$ (see \eqref{eq:blockHg}), which leads to a better convergence rate
when the learning objective satisfies certain conditions.
Specifically, we first consider the situation when the sub-problem \eqref{eq:quadratic} is solved to optimality every time.
We demonstrate that when the objective function satisfies the
Kurdyka-{\L}ojasiewicz inequality \citep{SL63a,SL93a,KK98a},\footnote{
	The Kurdyka-{\L}ojasiewicz inequality is much weaker than strong
convexity.}
the proposed algorithm converge globally linearly.
Next, we show that even if the sub-problem is solved only
approximately, global linear convergence is still retained under the
same condition.
Then, based on the relation \eqref{eq:w}, we show that as long as the
dual objective converges globally linearly, so does the primal
objective.
Finally, we investigate the choice of $B_t$ and provide an explanation on
why the special choice of \eqref{eq:Ht} (in particular with $a_1^t >
0$ and $a_2^t$ small) improves the convergence rate and therefore the
communication complexity.

Throughout this section, we assume that either of the following holds.
\begin{assumption}
	\label{assum:LipGrad}
The function $\bxi$ is differentiable and its gradient is
$\rho$-Lipschitz continuous for some $\rho > 0$.
That is,
\begin{equation*}
\|\nabla \bxi(\bz_1) - \nabla \bxi(\bz_2)\| \leq \rho \| \bz_1 -
\bz_2\|, \quad \forall \bz_1, \bz_2.
\end{equation*}
\end{assumption}

\begin{assumption}
	\label{assum:Lipsloss}
The function $\bxi$ is $L$-Lipschitz continuous for some $L>0$.
\begin{equation*}
|\bxi(\bz_1) - \bxi(\bz_2)| \leq L \|\bz_1 - \bz_2\|,
\quad \forall \bz_1, \bz_2.
\end{equation*}
\end{assumption}

These assumptions on $\bxi$ are less strict than
requiring each sub-component $\xi_i$ to satisfy certain properties.

\subsection{Convergence Analysis when Sub-Problems are Solved Exactly}
We assume the following condition based on the
Kurdyka-{\L}ojasiewicz (KL) inequality \citep{SL63a,SL93a,KK98a} holds.
\begin{assumption}
\label{assum:dualstrong}
The objective function in the dual problem \eqref{eq:dual}
satisfies the Kurdyka-{\L}ojasiewicz inequality with exponent
$1/2$ for some $\mu > 0$.  That is,
\begin{equation}
	f(\AL) - f^* \leq
	\frac{\min_{\hat{\bs} \in \partial f(\AL)}\|\hat{\bs}\|^2 }{2\mu}
	= \frac{\min_{\bs \in \partial \bxi^*(-\AL)}\|\nabla G(\AL) +
\bs\|^2}{2 \mu},
	\forall \AL \in \Omega.
	\label{eq:strong}
\end{equation}
where $f^*$ is the optimal objective value of the dual problem
\eqref{eq:dual},
and $\partial \bxi^*(-\AL)$ is the set of sub-differential of $\bxi^*$
at $-\AL$.
\end{assumption}
The following lemma shows that
Assumption~\ref{assum:LipGrad}
implies Assumption~\ref{assum:dualstrong}.
\begin{lemma}
	\label{lemma:strong}
Consider the primal problem \eqref{eq:primal}.
If Assumption \ref{assum:LipGrad} holds, then the dual problem
satisfies \eqref{eq:strong} with $\mu = 1/\rho$.
\end{lemma}

We start from showing that the update direction is indeed a descent
direction and Algorithm~\ref{alg:linesearch} terminates within a
bounded number of steps.
\begin{lemma}
	\label{lemma:linesearch}
	If $B_t$ is chosen so that the smallest eigenvalue of $B_t$ is no
	smaller than some constant $C_1$ (which can be
	nonpositive) for all $t$, and $Q^{\AL^t}_{B_t}$ is $C_2$-strongly
	convex for some $C_2>0$ for all $t$ and $C_1 + C_2 > 0$, then the
	update direction obtained by solving \eqref{eq:quadratic} exactly
	is a descent direction, and Algorithm~\ref{alg:linesearch}
	terminates in finite steps, with the generated step size lower
	bounded by
\begin{equation*}
\eta_t \geq \min\left(1, \frac{\beta(1 - \tau) \sigma \left( C_1 + C_2
\right)}{\left\|X^T X\right\|}\right), \forall t.
\end{equation*}
\end{lemma}

In contrast to most Newton-type methods such as
\cite{PT07a,JL14a} that require $B_t$ to be positive definite, the
conditions required in Lemma \ref{lemma:linesearch} are weaker, as even
if $B_t$ is not positive definite,
$Q_{B_t}^{\AL^t}$ can be strongly convex when
Assumption \ref{assum:LipGrad} holds.
As our framework is more general,
we allow a broader choice of $B_t$.  In Lemma
\ref{lemma:linesearch}, consider the choice of $B_t$ in \eqref{eq:Ht},
since $\tilde{H}_{\AL_t}$ is positive semidefinite, we have that $C_1
= a^t_2$.  For $C_2$, if Assumption \ref{assum:LipGrad} holds, then
since $\bxi^*$ is $(1/\rho)$-strongly convex, we have that $C_2 = C_1
+ 1/\rho$, and otherwise $C_2 = C_1$.

Now, we are ready to show the global linear convergence of the proposed 
Algorithm \ref{alg:blockcd} for solving \eqref{eq:dual}.
\begin{theorem}
	\label{thm:duallinear}
	If Assumption \ref{assum:dualstrong} holds, there exists $C_3 > 0$
	such that $\|B_t\|\leq C_3$ for all $t$, and that the conditions
	in Lemma \ref{lemma:linesearch} are satisfied for all iterations
	for some $C_2$ and $C_1 \leq C_3$, then the sequence of dual
	objective values generated by Algorithm \ref{alg:blockcd}
	converges Q-linearly to the optimum, with a rate of
	\begin{align*}
		&~\frac{f(\AL^{t+1}) - f^*}{f\left( \AL^t \right) - f^*}\\
		\leq &~1 - \frac{\mu \left( C_1 + C_2 \right) \tau}{\mu \left(
			C_1 + C_2 \right) \tau + 2\left( \frac{\|X^T
			X\|^2}{\sigma^2} + C_3^2 \right)}\min\left\{1, \frac{\beta
			\left( 1 - \tau \right)\sigma \left( C_1 + C_2
		\right)}{\|X^T X\| } \right\}, \forall t.
	\end{align*}
\end{theorem}

\subsection{Convergence Analysis when Sub-Problems are Solved Approximately}
In practice, when $B_t$ is not diagonal, the sub-problem
\eqref{eq:quadratic} is usually solved by an iterative solver and it
is time consuming to obtain an exact solution.
In this subsection, we show that Algorithm~\ref{alg:blockcd} still
converges linearly with inexact sub-problem solutions.
Our analysis is based on that in \cite{CPL18a,WP18a} to
assume that \eqref{eq:quadratic} is solved
$\gamma$-approximately for some $\gamma \in [0,1)$, defined below.
\begin{definition}
We say that $\Delta \AL^t$ solves \eqref{eq:quadratic}
$\gamma$-approximately for some $\gamma$ if
\begin{equation}
Q_{B_t}^{\AL^t}(\DAL^t) - \min_{\DAL}\,
Q_{B_t}^{\AL^t}\left(\DAL\right)
\leq
\gamma \left( Q_{B_t}^{\AL^t}\left( \bzero \right) - \min_{\DAL}\,
Q_{B_t}^{\AL^t}\left(\DAL\right)\right).
\label{eq:approx}
\end{equation}
\end{definition}

We will show that linear convergence can be obtained as long as the problem
\eqref{eq:dual} is convex and the quadratic growth condition holds.
\begin{assumption}
\label{assum:dualstrong2}
The function $f$ in the dual problem \eqref{eq:dual}
satisfies the quadratic growth condition with some
$\mu > 0$.
That is, let $A$ be the solution set, then
\begin{equation}
	f\left(\AL\right) - f^* \geq \frac{\mu}{2} \min_{\AL^* \in A} \|
	\AL - \AL^*\|^2, \forall \AL.
	\label{eq:strong2}
\end{equation}
\end{assumption}
Notice that \cite[Theorem 5]{JB15a} has shown that \eqref{eq:strong2}
and \eqref{eq:strong} are equivalent when $f$ is convex. Here we use
this equivalent condition for the ease of the convergence proof.

We are now able to present the convergence results of the inexact version
of our algorithm.
\begin{lemma}
\label{lemma:linesearch2}
If $B_t$ is chosen so that the smallest eigenvalue of $B_t$ is no
smaller than some constant $C_1$ (can be nonpositive) for all $t$,
$Q^{\AL^t}_{B_t}$ is $C_2$-strongly
convex for some $C_2>0$ for all $t$, and $(1 + \sqrt{\gamma})C_1 + (1 -
\sqrt{\gamma}) C_2 > 0$, then the update direction obtained by solving
\eqref{eq:quadratic} at least $\gamma$-approximately for some $\gamma
\in [0,1)$ is a descent direction, and Algorithm~\ref{alg:linesearch}
	terminates in finite steps, with the generated step size lower
	bounded by
\begin{equation}
\eta_t \geq \min\left(1, \frac{\beta(1 - \tau) \sigma \left( \left( 1
	+ \sqrt{\gamma} \right)C_1 + \left( 1 - \sqrt{\gamma} \right)C_2
\right)}{\left\|X^T X\right\| \left( 1 + \sqrt{\gamma}
\right)}\right).
\label{eq:stepsize}
\end{equation}
\end{lemma}

\begin{theorem}
\label{thm:duallinear2}
If Assumption~\ref{assum:dualstrong2} holds, there exists $C_3 > 0$
such that both $\|B_t\|\leq C_3$ and the conditions in
Lemma \ref{lemma:linesearch2} are satisfied for all $t$ for
some $C_1 \in [0,C_3]$, $C_2$, and $\gamma \in [0,1)$,
then the dual objective sequence generated by Algorithm~\ref{alg:blockcd}
converges Q-linearly as follows.
\begin{align}
\nonumber
&~\frac{f\left( \AL^{t+1} \right) - f^*}{f \left( \AL^t \right) -
f^*}
\\
\label{eq:linear}
\leq
&~\begin{cases}
1 - \frac{\tau \mu}{4C_3}
\min\left\{ 1 - \gamma, \frac{\beta \left( 1 - \tau \right)\sigma
	\left( \left( 1 - \gamma \right)C_1 + \left( 1 - \sqrt{\gamma}
\right)^2 C_2\right) }{\left\|X^T X\right\|} \right\}, \text{ if }
\mu \le 2 C_3,\\
1 - \tau
\left( 1 - \frac{C_3}{\mu} \right)
\min\left\{ 1 - \gamma, \frac{\beta \left( 1 - \tau \right)\sigma
	\left( \left( 1 - \gamma \right)C_1 + \left( 1 - \sqrt{\gamma}
\right)^2 C_2\right) }{\left\|X^T X\right\|} \right\}, \text{else}.
\end{cases}
\end{align}
\end{theorem}

\subsection{Convergence of the Primal Objective Using $\bw(\AL)$}
Next, we show the convergence rate of the primal problem \eqref{eq:primal}
based on the above linear convergence results for the dual problem. 
Our analysis here needs neither Assumption \ref{assum:dualstrong} nor
Assumption \ref{assum:dualstrong2} to hold.
The following theorem is obtained from the duality gap guarantees of
the algorithms in \cite{FB15a,SSS12b}, through taking the dual iterates
generated by Algorithm~\ref{alg:blockcd} and their corresponding primal
iterates in \eqref{eq:w} as the initial point for their algorithms.

\begin{theorem}
	\label{thm:dualitygap}
	For any $\epsilon > 0$ and any
	$\epsilon$-accurate solution $\AL$ for \eqref{eq:dual},
	the $\bw$ obtained through \eqref{eq:w} is:
	\begin{enumerate}
		\item
	$( \epsilon(1+ \rho\|X^TX\| / \sigma))$-accurate for
	\eqref{eq:primal}, if Assumption \ref{assum:LipGrad} holds, or
	\item $(\max\{2 \epsilon, \sqrt{8\epsilon \|X^T X\|
			L^2 / \sigma}\})$-accurate for \eqref{eq:primal}, if
			Assumption~\ref{assum:Lipsloss} holds.
	\end{enumerate}
\end{theorem}

By noting that $\log \sqrt{1 / \epsilon} = \log(1 / \epsilon) / 2$,
we get the following corollary.
\begin{corollary}
	\label{cor:primallinear}
	If we apply Algorithm \ref{alg:blockcd} to solve a regularized ERM
	problem that satisfies either Assumption \ref{assum:LipGrad} or
	Assumption \ref{assum:Lipsloss}, and the dual objective at the
	iterates $\AL^t$ converges Q-linearly to the optimum,
	then the primal objective evaluated at the iterates $\bw^t$
	obtained from the dual iterates $\AL^t$ via \eqref{eq:w} converges
	R-linearly to the optimum at the rate given by Theorem~\ref{thm:duallinear2}.
\end{corollary}

\subsection{Convergence Improvement by Using the Specific Choice of
$B_t$}
\label{subsec:improved}
Our analysis so far does not consider the effect of the specific
choices of \eqref{eq:Ht} in the quadratic approximation.
These results provided a worst-case guarantee that ensures the proposed
algorithms are provably efficient even for ill-conditioned problems.
In the following, we demonstrate that in most cases, the choice of
$B_t$ in \eqref{eq:Ht} leads to a better convergence rate because it
leverages the curvature information of the original problem to improve
the problem condition.

In particular, we assume the following.

\begin{assumption}
\label{assum:improved}
In \eqref{eq:dual}, the smooth term $G^*$ is $L_B$-Lipschitz
continuously differentiable with respect to the seminorms induced by the
matrices $B_t$ defined in \eqref{eq:Ht} around the point $\AL^t$:
\begin{gather}
\nonumber
G^*\left( \AL \right) \leq G^* \left( \AL^t \right) +
\nabla G^*(\AL^t)^T \left( \AL - \AL^t \right) + \frac{L_B}{2}
\|\AL - \AL^t\|^2_{B_t},\\  \forall \AL \in \Omega \text{
close enough to $\AL^t$}, \forall t.
\label{eq:betterLip}
\end{gather}
The objective function $f$ has quadratic growth with factor $\mu_B$ with
respect to the same seminorms:
\begin{equation}
	f\left( \AL \right) - f^* \geq \frac{\mu_B}{2} \min_{\AL^* \in A}
	\left\|\AL - \AL^* \right\|^2_{B_t}, \forall t, \forall \AL \in
	\Omega.
\label{eq:betterQG}
\end{equation}
Without loss of generality, we assume $\mu_B \leq 2$.
\end{assumption}

The assumption states that using the matrix $B_t$ defined in
\eqref{eq:Ht} gives a better problem condition under the norm change
(note that $\|X^T X\| / (\mu \sigma)$ is the condition number of the
dual problem under the Euclidean norm).
The Lipschitz continuity part is only needed locally as in our
convergence proof, this constant is only used for the step size bound,
which is for the local behavior in the region between the current and
the next iterates.
When $G^*$ is quadratic, i.e., when $g(\cdot) = \|\cdot\|^2 /2 $ in
\eqref{eq:primal}, $B_t$ becomes a fixed matrix if $a_1^t$ and $a_2^t$
are fixed over $t$, and this seminorm definition is more intuitive and
can be verified easily.
In particular, in this case $\nabla^2 G (\AL) \equiv X^T
X$, so
\begin{equation}
	\begin{cases}
	\nabla^2 G &\preceq K \tilde{H}_{\AL}, \forall \AL,\\
	\nabla^2 G &\preceq \|X^T X\| I,
\end{cases}
\label{eq:bdd}
\end{equation}
where the first inequality is from the observation that
\begin{equation}
	\label{eq:reasoning}
	\left\|\sum_{i=1}^l \AL_i X_i\right\|^2 \leq K \sum_{k = 1}^K
	\left\|\sum_{i
	\in J_k} \AL_i X_i\right\|^2.
\end{equation}

We can now show the improved convergence results.
We start from that the preconditioner can increase the step size.
\begin{lemma}
\label{lemma:improvedline}
Given $\AL^t$, if we use $B_t$ defined in
\eqref{eq:Ht} to form the sub-problem \eqref{eq:quadratic}, the
sub-problem is solved at least $\gamma$-approximately for some $\gamma
\in [0,1)$, and Assumption~\ref{assum:improved} holds, then the
	obtained update direction $\Delta \AL^t$ satisfies
\begin{equation}
	\Delta_t \leq - \frac{1}{1 + \sqrt{\gamma}} \left\|\Delta
	\AL^t\right\|^2_{B_t}.
\label{eq:deltabdd}
\end{equation}
Moreover, given $\beta,\tau \in (0,1)$,
Algorithm~\ref{alg:linesearch} produces a step size lower-bounded by
\begin{equation*}
	\eta_t \ge \min \left\{1, \frac{2 \beta ( 1 -
	\tau)}{L_B (1 + \sqrt{\eta})}\right\}.
\end{equation*}
\end{lemma}
If $L_B$ is much smaller than $L$, Lemma~\ref{lemma:improvedline}
provides a step size larger than what we had from the general analysis.
By utilizing the inequalities in \eqref{eq:bdd} for an upper bound on
$L_B$, we can see that usually a step size no smaller than $a_1^t/K$
can be expected if $a_1^t \in [0,K]$.

Note that we can take the larger of the bound from
Lemma~\ref{lemma:linesearch2} and that from
Lemma~\ref{lemma:improvedline}, so we are always guaranteed to have a
bound that is no worse.
We therefore assume without loss of generality that the bound from
Lemma~\ref{lemma:improvedline} is the larger one.

We proceed on to the improved convergence speed on the dual problem.
\begin{theorem}
\label{thm:improved}
If Assumption \ref{assum:improved} holds,
and that the conditions
in Lemma~\ref{lemma:improvedline} are satisfied for all iterations,
then the sequence of dual objective values generated by Algorithm
\ref{alg:blockcd} with $B_t$ defined in \eqref{eq:Ht}
converges Q-linearly to the optimum, with the rate being
\begin{align}
\label{eq:improved0}
\frac{f\left(\AL^{t+1} \right) - f^*}{f\left( \AL^t \right) - f^*}&\leq
1 - \frac{\tau \mu_B ( 1 - \gamma)\eta_t}{4}\\
&\leq 1 - \frac{\tau \mu_B}{2} \min\left\{ frac{1 - \gamma}{2}, \frac{\beta
	\left( 1 - \tau \right) \left( 1 - \sqrt{\gamma} \right)}{L_B}
\right\}.
	\label{eq:improved}
\end{align}
\end{theorem}
When $L_B / \mu_B$ is much smaller than the condition number $\|X^T
X\| /(\mu\sigma)$ defined by the Euclidean norm, which is usually the
case as $B_t$ approximates the Hessian closely, the rate in
\eqref{eq:improved} is significantly faster than the best possible
rates in Theorems~\ref{thm:duallinear} and \ref{thm:duallinear2}
obtained by making the algorithm the proximal gradient method with a
fixed step size.
Finally, by combining Theorem~\ref{thm:improved} and
Theorem~\ref{thm:dualitygap}, we get a faster convergence rate for the
primal objective as well.

\begin{corollary}
If we apply Algorithm \ref{alg:blockcd} with $B_t$ defined in
\eqref{eq:Ht} to solve a regularized ERM problem that satisfies either
Assumption \ref{assum:LipGrad} or Assumption \ref{assum:Lipsloss}, and
Assumption~\ref{assum:improved} holds,
then the primal iterates $\bw(\AL^t)$ obtained from the dual iterates
$\AL^t$ via \eqref{eq:w} give primal objectives that converge to the
optimum R-linearly at the rate given by Theorem~\ref{thm:improved}.
\end{corollary}

\section{Related Works}
\label{sec:related}
The general convergence theory of the inexact version of our general
framework in Section~\ref{sec:method} follows from \cite{CPL18a,WP18a}
on the line of inexact variable metric methods for regularized
optimization.
The analysis of the exact version, derived independent of the theory
in \cite{CPL18a,WP18a}, is applicable to a broader choice of
sub-problems, in the sense that indefinite $B_t$ is allowed in the
exact version.

On the other hand, our focus is on how to devise a good approximation
of the Hessian matrix of the smooth term that makes
distributed optimization efficient.
Works focusing on this direction for dual ERM problems include
\cite{DP11a,TY13b,CM15b}.
\citet{DP11a} discusses how to solve the SVM dual problem
in a distributed manner.
This problem is a special case of \eqref{eq:dual}; see
Section \ref{subsec:binary} for more details.
They proposed a method called \dsvm that iteratively solves
\eqref{eq:quadratic} using the $B_t$ defined in \eqref{eq:Ht} with
$a_1^t \equiv 1, a_2^t \equiv 0$ to obtain the update direction,
while the step size $\eta_t$ is fixed to $1/K$.
Though they did not provide theoretical convergence guarantee in
\cite{DP11a},
with the understanding the SVM dual problem is quadratic, our
analysis in Lemma~\ref{lemma:improvedline} and the bound in
\eqref{eq:bdd} gives convergence guarantee for their choice.

In \cite{TY13b}, the algorithm \disdca is proposed to solve
\eqref{eq:dual} under the assumption that $g$ is strongly convex.
They consider the case $c_i \equiv 1$ for all $i$, but the algorithm
can be directly generalized to $c_i > 1$.
\disdca specifically uses
the stochastic dual coordinate descent (SDCA) method \citep{SSS12a} to
solve the local sub-problems,
while the choice of $B_t$ is picked according to the algorithm parameters.
To solve the sub-problem on machine $k$, each time SDCA samples one
entry $i_k$ from $J_k$ with replacement and minimizes the local
objective with respect to $\AL_{i_k}$.
At each iteration of their algorithm,
each time machine $k$ selects $m_k$ entries in uniform random to form
the sub-problem, and let us denote $m \coloneqq \sum_{k=1}^K m_k$.
The first variant of \disdca, called the basic variant in
\cite{TY13b}, sets $B_t$ in \eqref{eq:quadratic} as
\begin{equation*}
	(B_t)_{i,j} = \begin{cases}
		\frac{m}{\sigma} \bx_i^T \bx_j &\text{ if $\bx_i,\bx_j$ are
			from the same $X_k$
		for some $k$ sampled},\\
		0 &\text{ else,}
	\end{cases}
\end{equation*}
and the step size is fixed to $1$.
In this case, it is equivalent to splitting the data into $l$ blocks,
and the minimization is conducted only with respect to the blocks selected.
If we let $I$ be the indices not selected, then following the same
reasoning for \eqref{eq:reasoning}, we have
\begin{equation}
	\|\bd\|^2_{\nabla^2 G^*(\AL)} \leq \bd^T B_t \bd, \quad \forall \bd
	\text{ such that
	} \bd_I = \bzero, \forall \AL,
	\label{eq:bound}
\end{equation}
Therefore, by \eqref{eq:bound} and the Lipschitz-continuous
differentiability of $G^*$, it is not hard to see that in this
case minimizing $Q^{\AL}_{B_t}$ directly results in a certain amount
of function value decrease.
The analysis in \cite{TY13b} then shows that the primal iterates
$\{\bw_t\}$ obtained by substituting the dual iterates $\{\AL_t\}$
into \eqref{eq:w} converges linearly to the optimum when all $\xi_i$
have Lipschitz continuous gradient and converges with a rate of
$O(1/\epsilon)$ when all $\xi_i$ are Lipschitz continuous
by using some proof techniques similar to that in \cite{SSS12b}.
As we noted in Section \ref{sec:analysis}, this is actually the
same as showing the convergence rate of the dual objective and then
relating it to the primal objective.

The second approach in \cite{TY13b}, called the practical variant,
considers
\begin{equation*}
	(B_t)_{i,j} = \begin{cases}
		\frac{K}{\sigma} \bx_i^T \bx_j &\text{ if } \pi(i) = \pi(j),\\
		0 &\text{ else,}
	\end{cases}
\end{equation*}
and takes unit step sizes.
Similar to our discussion above for their basic variant,
$Q^{\AL}_{B_t}$ in this case is also an upper bound for the function
value decrease if the step size is fixed to $1$,
and we can expect this method to work better than the basic variant as
the approximation is closer to the real Hessian and the scaling
factor is closer to one.
Empirical results show that this variant is as expected faster than
the basic variant,
despite the lack of theoretical convergence guarantee in \cite{TY13b}.

Both \dsvm and the practical variant of \disdca are generalized to a
framework proposed in \cite{CM15b} that discusses the relation between
the second-order approximation and the step size.
In particular, their theoretical analysis for fixed step sizes starts
from \eqref{eq:bdd} and the Lipschitz continuous differentiability of
$G^*$ to form safe upper bounds for the function value decrease.
They considered solving \eqref{eq:quadratic} with $B_t$ defined as
\begin{equation}
	(B_t)_{i,j} = \begin{cases}
		\frac{a}{\sigma} \bx_i^T \bx_j &\text{ if } \pi(i) = \pi(j),\\
		0 &\text{ else,}
	\end{cases}
	\label{eq:B}
\end{equation}
and showed that for $a \in [1,K]$, a step size of $a / K$ is enough
to ensure convergence of the dual objective,
similar to our result in Lemma~\ref{lemma:improvedline}.
As we discussed above for \cite{DP11a} and \cite{TY13b}, this choice
can be proven to ensure objective value decrease.
Unlike \disdca which is tied to SDCA,
their framework allows arbitrary solver for the local sub-problems,
and relates how precisely the local sub-problems are solved with
the convergence rate.
If we ignore the part of local precision,
the convergence rates of their framework shown in \cite{CM15b} is
similar to that of \cite{TY13b} for the basic variant of \disdca.
This work therefore provides theoretical convergence rates
similar to the basic variant of \disdca for both \dsvm and the
practical variant of \disdca,
and their experimental results shows that the practical variant of
\disdca is indeed the most efficient.
Notice that despite empirically bettering the proximal gradient
method, these works all provide convergence rates no better than it.
Fortunately, as these methods can be seen as special cases of our framework,
with our analysis in Section \ref{subsec:improved}, we can explain why
they are all faster than the proximal gradient method.
Our analysis in Section~\ref{subsec:improved} is, up to our best
knowledge, a novel result that shows how improved convergence speed
can be obtained when the second-order approximation is selected
properly.

When we use the $B_t$ considered by those works in
\eqref{eq:quadratic}, the major algorithmic difference is that we do
not take a pre-specified safe step size.
Instead, we dynamically find a possibly larger step size that,
according to \eqref{eq:improved0}, can provide more function decrease
than directly applying the lower bound in
Lemma~\ref{lemma:improvedline}.
We can see that the choice of $a = 1$ in \eqref{eq:B} gives too
conservative the step size, while the choice of $a=K$ might make the
quadratic approximation in \eqref{eq:quadratic} deviate from the real
Hessian too far.
In particular, assuming $\nabla^2 g \equiv I / \sigma$, the case of $a=1$
makes the Frobenius norm of the difference between $\nabla^2 G^*$ and
$B_t$ the smallest, while other choices increase this value.
This suggests that using $a=1$ should be the best approximation one
can get,
but even directly using $\nabla^2 G^*$ might not guarantee decrease of
the objective value.
Our method thus provides a way to deal with this problem by adding a
low-cost line search step.
Moreover, by adding Assumption~\ref{assum:dualstrong2} that holds true
for most ERM losses (see discussion in the next section), we are able
to show linear convergence for a broader class of problems.

Most other distributed ERM solvers directly optimize
\eqref{eq:primal}.
Primal batch solvers for ERM that require
computing the full gradient or Hessian-vector products are
inherently parallelizable and can be easily extended to
distributed environments as the main
computations are matrix-vector products like $X^T\bw$.
It mostly takes only some implementation modifications to make
these approaches efficient in distributed environments.
Among them, it has been shown that distributed truncated-Newton
\citep{YZ14a,CYL14a}, distributed limited-memory BFGS \citep{WC14a},
and the limited-memory common-directions method \citep{CPL17a} are the
most efficient ones in practice.
These methods have the advantage that their convergences are invariant
of the data partition, though with the additional requirement that the primal
objective is differentiable or even twice-differentiable in comparison
with ours.
However, there are important cases of ERM problems that do not possess
differentiable primal objective function such as the SVM problem.
In these cases, one still needs to consider the dual approaches, for
otherwise the convergence might be extremely slow.
Another popular distributed algorithm for solving \eqref{eq:primal}
without requiring differentiability is the alternating direction
method of multipliers (ADMM) \citep{SB11b}, which is widely used in
consensus problems.
However, it has been shown in \cite{TY13b} and \citet{YZ14a} that
\disdca and truncated-Newton outperforms ADMM on various ERM problems.

Many lately proposed distributed optimization methods
focus on the communication efficiency.
By increasing the computation per iteration, they are able to use fewer
communication rounds to obtain the same level of objective value.
However, these approaches either rely on the stronger assumption that
data points across machines are independent and identically
distributed (i.i.d.), or has higher computational dependency on the
dimensionality of the problem.
The former assumption may not hold in practice, while the latter
results in computational-inefficient and thus impractical methods.
For example, \citet{YZ15a} consider a damping Newton method with a
preconditioned conjugate gradient (PCG) method to solve the Newton
linear system, with the preconditioner being the Hessian from a
specific machine.
However, the computational cost of PCG is much higher because each
iteration of which involves inverting the local Hessian.
The distributed SVRG method proposed by \citet{JDL15a} has
good computational and communication complexity simultaneously, but
needs the assumption that data points on each machine follow a
certain distribution and requires overlapping data points on
different machines. This implies more communication in
advance to distribute data points, which is prohibitively
expensive when the data volume is huge.
Indeed, instead of a real distributed environment, their experiment is
simulated in a multi-core environment because of this constraint.
We therefore exclude comparison with these methods.

Recently, \citet{ZS17a} adopted for \disdca an acceleration technique
in \cite{HL15a,SSS13a}, resulting in a theoretically faster algorithm.
This technique repeatedly uses \disdca to solve a slightly modified
objective every time to some given precision, and reconstruct a new
objective function based on the obtained iterate and the previous iterate.
The same technique can also be applied to this work in the same
fashion by replacing \disdca with the proposed method.
Therefore, we focus on the comparison with methods before applying the
acceleration technique, with the understanding that the faster method
of the same type before acceleration will result in a faster method
after acceleration as well.
Moreover, what is the best way to apply the acceleration technique to
obtain the best efficiency for distributed optimization of ERM
problems is itself another open research problem.
Issues including whether to apply it on the primal or the dual problem,
should restarting be considered, how to estimate the unknown
parameters, and so on, are left to future work.

\section{Applications}
\label{sec:applications}
In this section, we apply the proposed algorithm in
Section~\ref{sec:method} to solve various regularized ERM problems and
discuss techniques for improving the efficiency by utilizing the
problem structures.
We will demonstrate the empirical performance of the proposed
algorithms in Section~\ref{sec:exp}.

We first show that a class of problems satisfies
Assumption~\ref{assum:dualstrong2}.
\begin{lemma}[{\cite[Theorem~10]{IN19a}}]
	\label{lemma:strong2}
	Consider a problem of the following form
	\begin{subequations}
		\label{eq:linear-strong}
		\begin{align}
		\min_{\AL} &\quad F(\AL)\coloneqq g(A \AL) + \bb^T
		\AL\label{eq:glinear}\\
		\text{subject to} &\quad C\AL \leq \bd,
		\label{eq:poly}
	\end{align}
	\end{subequations}
	where $g$ is strongly convex
	with any feasible initial point $\AL^0$.
	Then $F$ satisfies the condition \eqref{eq:strong} in the level
	set $\{\AL\mid C\AL \leq \bd,\quad F(\AL) \leq F(\AL^0)\}$ for
	some $\mu > 0$ that depends on the initial point $\AL^0$.
	If the constraint is a polytope, then the condition
	\eqref{eq:strong} holds for all feasible $\AL$.
\end{lemma}

\subsection{Binary Classification and Regression}
\label{subsec:binary}
The first case is the SVM problem~\citet{BB92a,VV95a}
where $c_i \equiv 1$,
and given $C>0$,
\begin{equation*}
	\xi_i(z) \equiv C\max(1 - y_i z, 0),\quad
	g(\bw) = \frac12 \|\bw\|^2,
\end{equation*}
with $y_i \in \{-1,1\}, \forall i$.
Obviously, Assumption \ref{assum:strong} holds with $\sigma =
1$ in this case,
and $\xi_i$ are $1$-Lipschitz continuous.
Based on a straightforward derivation, we have that
\begin{equation}
	g^*(X\AL) = \frac12 \|X \AL\|^2,\quad
	\xi^*_i(-\alpha_i) = \indicator_{[0,C]}(\alpha_i y_i) - \alpha_i y_i,
\end{equation}
where
\[
	\indicator_{[0,C]}(x) \coloneqq
	\begin{cases}
	0 & \text{ if } x \in [0,C],\\
	\infty & \text{ else}.
\end{cases}
\]
It is clear that $\xi_i$ are $C$-Lipschitz continuous.
For the dual problem, we see that the constraints are of the form
\eqref{eq:poly}, $g^*$ is strongly convex with respect to $X\AL$, and
the remaining term $-\alpha_i y_i$ is linear, so the objective
function satisfies the form \eqref{eq:glinear}.
Therefore,
by Lemma \ref{lemma:strong}, \eqref{eq:strong} is satisfied.
Therefore all conditions of Assumption \ref{assum:Lipsloss} are
satisfied.
Hence from Corollary \ref{cor:primallinear}, our algorithm enjoys
linear convergence in solving the SVM dual problem.

Besides, we can replace the hinge loss (L1 loss) in SVM with the
squared-hinge loss (L2 loss):
\begin{equation*}
	\xi_i(z) \equiv C \max(1 - y_i z, 0)^2,
\end{equation*}
and then $\xi_i$ becomes differentiable, with the gradient being
Lipschitz continuous.
Therefore, Assumption \ref{assum:LipGrad} is satisfied,
and we can apply Corollary \ref{cor:primallinear}.
We have that
\begin{equation*}
	\xi^*_i(-\alpha_i) = \indicator_{[0,\infty)}(\alpha_i y_i) -
		\alpha_i y_i + \frac{\alpha_i^2}{4C}.
\end{equation*}

One can observe that the dual objectives of the hinge loss and
the squared-hinge loss SVMs are both quadratic,
hence we can apply the exact line search approach in
\eqref{eq:quadlinesearch} with very low cost by utilizing the $\Delta
\bv$ and $\bv$ vectors.

Another widely used classification model is logistic regression,
where
\begin{equation*}
	\xi_i(z) \coloneqq C \log(1 + \exp(-y_i z)).
\end{equation*}
It can then be shown that the logistic loss is infinitely differentiable, 
and its gradient is Lipschitz continuous.
Thus, Assumption \ref{assum:LipGrad} is satisfied.

An analogy of SVM to regression is support vector regression (SVR)
by \citet{BB92a,VV95a} such that the $g$ function is the same and
given $C> 0$ and $\epsilon \geq 0$,
\begin{equation*}
	\xi_i(z) \coloneqq \begin{cases}
		C \max(|z - y_i| - \epsilon, 0),& \text{or }\\
		C \max(|z - y_i| - \epsilon, 0)^2,
	\end{cases}
\end{equation*}
with $y_i\in\R$ for all $i$.
Similar to the case of SVM, the first case satisfies Assumption
\ref{assum:Lipsloss}\footnote{Up to an equivalent reformulation of the
	dual problem by setting $\AL = \AL^+ - \AL^-$ and $\AL^+,\AL^-
	\geq 0$.}
and the latter satisfies Assumption \ref{assum:LipGrad}.
Often the first variant is called SVR and the second variant is called
L2-loss SVR.
Note that the degenerate case of $\epsilon = 0$ corresponds to the
absolute-deviation loss and the least-square loss.
In the case of the least-square loss, we again can use the exact line
search approach because the objective is a quadratic function.

Note that one can also replace $g$ with other strongly convex
functions, but it is possible that \eqref{eq:strong} is not satisfied.
In this case, one can establish some sublinear convergence
rates by applying similar techniques in our analysis, but we omit
these results to keep the description straightforward.

A short summary of various $\xi_i$'s we discussed in this section is
in Table \ref{tbl:losses}.

\begin{table}
\centering
    \begin{tabular}{@{}l@{ }|@{ }r@{ }r@{ }r@{}}
	Loss name  &     $\xi_i(\bz)$  &
	Assumption &  $\xi^*_i(-\AL)$\\
	\hline
	L1-loss SVM &  $C \max(1 - y_i z, 0)$  &    \ref{assum:Lipsloss} &
	$\indicator_{[0,C]}(\alpha y_i) - \alpha y_i$  \\
	L2-loss SVM &  $C \max(1 - y_i z, 0)^2$  &  \ref{assum:LipGrad} &
	$\indicator_{[0,\infty)}(\alpha y_i) - \alpha y_i + \frac{\alpha^2}{4C}$ \\
		\multirow{2}{*}{logistic regression}   &  \multirow{2}{*}{$C
		\log(1 + \exp(-y_i z))$}   &
		\multirow{2}{*}{\ref{assum:LipGrad}} &  $\indicator_{[0,C]}(\alpha y_i)+ \alpha y_i \log (\alpha y_i) $  \\
	&&&
	$+ (C - \alpha y_i) \log(( C - \alpha y_i))$ \\
	\hline
	SVR    &  $C \max(|z - y_i| - \epsilon, 0)$  & 
	\ref{assum:Lipsloss} &	$\indicator_{[-C,C]}(\alpha) + \epsilon |\alpha_i| - \alpha y_i$  \\
	L2-loss SVR &  $ C \max(|z - y_i| - \epsilon, 0)^2 $  & 	\ref{assum:LipGrad} & 	$ \epsilon |\alpha_i| - \alpha y_i + \frac{1}{4C}\alpha^2 $  \\
	Least-square regression   &  $ C (z - y_i)^2 $    & 	\ref{assum:LipGrad} &  $ -\alpha	y_i + \frac{1}{4C}\alpha^2 $ \\
	\end{tabular}
	\caption{Summary of popular ERM problems for binary classification (the 
	range of $y=\{1,-1\}$) and for regression (the range of $y=R$), where 
	our approach is applicable. Our approach is also applicable to the 
	extensions of these methods for multi-class classification and structured 
	prediction.}
	\label{tbl:losses}
\end{table}

\subsection{Multi-class Classification}
\label{subsec:multiclass}
For the case of multi-class classification models, we assume without loss
of generality that $c_i \equiv T$ for some $T > 1$, and $y_i \in
\{1,\dotsc,T\}$ for all $i$.
The first model we consider is the multi-class SVM model proposed by
\citet{KC02d}.
Given an original feature vector $\bx_i \in \R^{\tilde{n}}$,
the data matrix $X_i$ is defined as $(I_T - \be_{y_i}
\b1^T) \otimes \bx_i $,
where $I_T$ is the $T$ by $T$ identity matrix,
$\be_i$ is the unit vector of the $i$-th coordinate, $\b1$ is the
vector of ones, and $\otimes$ denotes the Kronecker product.
We then get that $n = T \tilde{n}$, and the multi-class SVM model uses
\begin{align}
	g(\bw) &\coloneqq \frac12 \|\bw\|^2,\nonumber\\
	\xi_i(\bz) &\coloneqq C \max\left(\max_{1\leq j \leq T} 1 - z_i, 0\right).
	\label{eq:multiclass}
\end{align}
From the first glance, this $\xi$ seems to be not even Lipschitz
continuous.
However, its dual formulation is
\begin{align}
	\min_{\AL}\quad& \frac{1}{2}\|X \AL\|^2 + \sum_{i=1}^l \sum_{j
	\neq y_i} (\AL_i)_j\nonumber\\
	\text{subject to}\quad& \AL_i^T \b1 = 0, i=1,\dotsc,l,
	\label{eq:mcsvm}\\
	& (\AL_i)_j \leq 0, \forall j \neq y_i, i=1,\dotsc,l,\nonumber\\
	& (\AL_i)_{y_i} \leq C,i=1,\dotsc,l,\nonumber
\end{align}
showing the boundedness of the dual variables $\AL$.
Thus, the primal variable $\bw(\AL) = X\AL$ also lies in a bounded
area.
Therefore, $\xi_i(X_i^T\bw(\AL))$ also has a bounded domain, indicating
that by compactness we can find $L\geq 0$ such that this continuous
function is Lipschitz continuous within this domain.
Moreover, the formulation \eqref{eq:mcsvm} satisfies the form
\eqref{eq:linear-strong}, so \eqref{eq:strong} holds by Lemma
\ref{lemma:strong}.
Thus, Assumption \ref{assum:Lipsloss} is satisfied.
Note that in this case the objective of \eqref{eq:mcsvm} is a
quadratic function so once again we can apply the exact line search
method on this problem.

As an analogy of SVM, one can also use the squared-hinge
loss for multi-class SVM \citep{CPL12a}.
\begin{equation}
	\xi_i(\bz) \coloneqq C \max\left(\max_{1\leq j \leq T} 1 - z_i,
	0\right)^2.
	\label{eq:multiclass2}
\end{equation}
The key difference to the binary case is that the squared-hinge loss
version of multi-class SVM does not possess a differentiable objective
function.
We need to apply a similar argument as above to argue the
Lipschitzness of $\xi$.
The dual formulation from the derivation in \cite{CPL12a} is
\begin{align*}
	\min_{\AL}\quad& \frac{1}{2}\|X \AL\|^2 + \sum_{i=1}^l \sum_{j
	\neq y_i} (\AL_i)_j + \sum_{i=1}^l \frac{((\AL_i)_{y_i})^2}{4C} \nonumber\\
	\text{subject to}\quad& \AL_i^T \b1 = 0, i=1,\dotsc,l,\\
	& (\AL_i)_j \leq 0, \forall j \neq y_i, i=1,\dotsc,l,\nonumber
\end{align*}
suggesting that each coordinate of $\AL$ is only one-side-bounded,
so this is not the case that $\AL$ lies explicitly in a compact set.
However, from the constraints and the objective, we can see that given
any initial point $\AL^0$, the level set $\{\AL\mid f(\AL) \leq
f(\AL^0)\}$ is compact.
Because our algorithm is a descent method, throughout the
optimization process, all iterates lie in this compact set.
This again indicates that $\bw(\AL)$ and $X_i^T \bw(\AL)$ are within a
compact set, proving the Lipschitzness of $\xi_i$ within this set.
The condition \eqref{eq:strong} is also satisfied following the same
argument for the hinge-loss case.
Therefore, Assumption \ref{assum:Lipsloss} still holds, and it is
obvious that we can use the exact line search method here as well.

We can also extend the logistic regression model to the multi-class
scenario.
The loss function, usually termed as multinomial logistic regression
or maximum entropy, is defined as
\begin{equation*}
	\xi_i(\bz) \coloneqq -\log \left( \frac{\exp(z_{y_i})}{\sum_{k=1}^T
\exp(z_k)}\right).
\end{equation*}
It is not hard to see that Assumption \ref{assum:LipGrad} holds for
this problem. For more details of its dual problem and an efficient
local sub-problem solver, interested readers are referred to
\cite{HFY10a}.

\subsection{Structured Prediction Models }
In many real-world applications, the decision process involves making multiple
predictions over a set of interdependent output variables, whose
mutual dependencies can be modeled as a structured object such as a linear chain, a tree, or a graph.
As an example, consider recognizing a handwritten word, where characters
are output variables and together form a sequence structure. It is important to consider the correlations between the predictions
of adjacent characters to aid the individual predictions of characters.
A family of models designed for such problems are called structured
prediction models. In the following, we discuss how to apply Algorithm \ref{alg:blockcd}
in solving SSVM \citep{IT05a,BT04a}, a popular structured prediction
model.

Different from the case of binary and multi-class classifications,
the output in a structured prediction problem is a set of
variables  $\byi \in \yset_i$, and $\yset_i$ is the set of all feasible
structures. The sizes of the input and the output variables are often
different from instance to instance.
For example, in the handwriting recognition problem, each element in $\by$ represents
a character and $\yset$ is the set of all possible words. Depending on
the number of characters in the words, the sizes of inputs and outputs
vary.

Given a set of observations $\{(\bx_i,\by_i)\}_{i=1}^l$,
SSVM solves
\begin{align}
	\min_{\bw, \bpsi} &\quad \frac12 \|\bw\|^2+ C \sum_{i=1}^l
	\ell(\psi_i) \nonumber\\
	\text{subject to} &\quad \bw^T \phi(\by, \by_i, \bx_i)
	\geq \Delta (\by_i, \by) - \psi_i, \quad \forall \by \in \yset_i,\quad i = 1,\dotsc,l,
	\label{eq:ssvm}
\end{align}
where
$C>0$ is a predefined parameter,
$\phi(\by,\by_i,\bx_i) = \Phi(\bx_i, \by_i) - \Phi(\bx_i, \by),$ and
$\Phi(\bx,\by)$ is the generated feature vector depending on
both the input $\bx$ and the output structure $\by$. By defining features 
depending on the output, one can encode the output structure into the model 
and learn parameters to model the correlation between output variables.   
The constraints in problem \eqref{eq:ssvm} specify that the difference between the score assigned to the correct output structure should be
higher than a predefined scoring metric $\Delta(\by,\by_i)\geq 0$ that represents the distance between output structures.
If the constraints are not satisfied, then a penalty term $\psi_i$ is introduced to the objective function, where
$\ell(\psi)$ defines the loss term. Similar to the binary and
multi-class classifications cases, common choices of the loss
functions are the L2 loss and the L1 loss.
The SSVM problem \eqref{eq:ssvm} fits in our framework,
depending on the definition of the features, one can define $X_i$
to encode the output $\by$. One example is to set every column of
$X_i$ as a vector of the form $\phi(\by,\by_i,\bx_i)$ with different
$\by \in \yset_i$, and let $\xi_i(X^T_i\bw)$ in problem
\eqref{eq:primal} be
\begin{equation}
	\xi_i(\bz) =  C\max_{\by \in \yset_i}\ell(\Delta (\by_i,
	\by) -  \bz_\by).
\label{eq:ssvmloss}
\end{equation}
Here, we use the order of $\by$ appeared in the columns of $X_i$ as the
enumerating order for the coordinates of $\bz$.

We consider solving problem \eqref{eq:ssvm} in its dual form~\citep{IT05a}.
One can clearly see the similarity between \eqref{eq:ssvmloss} and
the multi-class losses \eqref{eq:multiclass} and \eqref{eq:multiclass2},
where the major difference is that the value $1$ in the multi-class
losses is replaced by $\Delta(\by_y,\by)$.
Thus, it can be expected that the dual problem of SSVM is similar to
that of multi-class SVM.
With the L1 loss,
the dual problem of \eqref{eq:ssvm} can be written as,
\begin{align}
	\min_{\AL}\quad& \frac12 \|X \AL\|^2 - \sum_{i=1}^l
	\sum_{\by\in \yset_i,\by \neq \by_i} \Delta(\by_i, \by) (\AL_i)_{\by_i}\nonumber\\
\label{eq:dual-ssvm1}
		\mbox{subject to} \quad&
		\AL_i^T \b1 = 0, i=1,\dotsc,l,\\
		&(\AL_i)_{\by} \leq 0, \forall \by \in \yset_i, \by \neq
		\by_i,i=1,\dotsc,l,\nonumber\\
		&(\AL_i)_{\by_i} \leq C, i=1,\dotsc,l.
		\nonumber
\end{align}
With the L2 loss, the dual of \eqref{eq:ssvm} is
\begin{align}
	\min_{\AL}\quad& \frac12 \|X \AL\|^2 - \sum_{i=1}^l
	\sum_{\by \neq \by_i} \Delta(\by_i, \by) (\AL_i)_{\by_i} +
	\sum_{i=1}^l \frac{((\AL_i)_{\by_i})^2}{4C}\nonumber\\
\label{eq:dual-ssvm2}
		\mbox{subject to} \quad&
		\AL_i^T \b1 = 0, i=1,\dotsc,l,\\
		&(\AL_i)_{\by} \leq 0, \forall \by \in \yset_i, \by \neq
		\by_i,i=1,\dotsc,l.\nonumber
\end{align}
As the dual forms are almost identical to that shown in Section
\ref{subsec:multiclass},
it is clear that all the analysis and discussion can be directly
used here.

	The key challenge of solving problems \eqref{eq:dual-ssvm1} and
	\eqref{eq:dual-ssvm2} is that for most applications,	the size
	of $\yset_i$ and thus the dimension of $\AL$  is exponentially
	large (with respect to the length of $\bx_i$),
so optimizing over all variables is unrealistic.
Efficient dual methods~\citep{IT05a,SLJ13a,ChangYi13}  maintain a small working set of dual
variables to optimize such that the remaining variables are fixed to be zero.
These methods then iteratively enlarge the working set until the
problem is well-optimized.\footnote{This approach is related to applying the cutting-plane methods
to solve the primal problem \eqref{eq:ssvm} \citep{IT05a,TJ08a}.}
The working set is selected using the sub-gradient of \eqref{eq:ssvm}
with respect to the current iterate.
Specifically, for each training instance $\bx_i$, we add the dual variable $\alpha_{i,\hat{\by}}$
corresponding to the structure $\hat{\by}$ into the working set, where
\begin{equation}
	\label{eq:loss-augmented-inference}
\hat{\by} = \arg\max_{\by \in \yset_i} \bw^T \phi(\by,\by_i,\bx_i) - \Delta(\by_i, \by).
\end{equation}
Once $\AL$ is updated, we update $\bw$ accordingly.
We call the step of computing
eq. \eqref{eq:loss-augmented-inference} ``inference'',
and call the part of optimizing Eq. \eqref{eq:dual-ssvm1} or \eqref{eq:dual-ssvm2} over a fixed working set ``learning''.
When training SSVM in a distributed manner, the learning step involves
communication across machines.  Therefore, inference and learning
steps are both expensive.  Our algorithm can be applied in the
learning step to reduce the rounds of communication, and linear
convergence rate for solving the problem under a fixed working set can
be obtained.

SSVM is an extension of multi-class SVM for structured prediction.
Similarly, conditional random fields (CRF)~\citep{LaffertyMcPe01}
extends multinomial logistic regression.
The loss function in CRF is defined as the negative log-likelihood:
\begin{equation}
	\xi_i(\bz) \coloneqq -\log \left(\frac{\exp(\bz_{\by_i})}
	{\sum_{\by\in \yset_i} \exp (\bz_{\by})}\right).
	\label{eq:crf}
\end{equation}
Similar to multinomial logistic regression,
Assumption \ref{assum:LipGrad} holds for \eqref{eq:crf}.

\section{Experiments}
\label{sec:exp}
We conduct experiments on different ERM problems to examine the
efficiency of variant realizations of our framework.
The problems range from binary classification (i.e., $c_i \equiv 1$) to
problems with complex output structures (i.e., each $c_i$ is different),
and from that exact line search can be conducted to that
backtracking using Algorithm \ref{alg:linesearch} is applied.
For each problem, we compare our method with the state of the art,
and the dataset is partitioned evenly across machines in terms of the
number of data points without randomly shuffling the instances in
advance, so it is possible that the data distributions on different
machines vary.

For the case of $c_i \equiv 1$, we consider two linear classification
tasks.
To evaluate the situation of larger $c_i$, we take SSVM as the
exemplifying application.

\subsection{Binary Linear Classification}
\label{subsec:binaryexp}
\begin{table}[tb]
\begin{center}
\caption{Data statistics.}
\label{tbl:data}
\begin{tabular}{@{}l|r|r|r@{}}
Data set & \#instances ($l$) & \#features ($n$) & \#nonzeros\\
\hline
\webspam & 350,000 & 16,609,143 & 1,304,697,446 \\
\uu & 2,396,130 & 3,231,961 & 277,058,644\\
\kddb & 19,264,097 & 29,890,095 & 566,345,888\\
\end{tabular}
\end{center}
\end{table}

The proposed framework is suitable for training large machine learning models 
on data where numbers of instances and features are both large. It is especially useful
in a practical setting where data instances are stored distributedly on multiple machines. 
This setting is common on web data due to efficiency and privacy concerns. 
We therefore consider the following large-scale datasets in our experiments. 
\begin{itemize}
    \item {\webspam}~\citep{wang2012evolutionary} is a binary classification task aiming at detecting if a web page is created to manipulate search engines. We use bag-of-words model with n-gram ($n\leq3$) to extract features.  
    \item {\uu}~\citep{ma2009identifying} detects malicious URLs based on their lexical and host-based features.
    \item {\kddb} is the dataset used in KDD Cup 2010 with the goal to predict students' performance based on logs of their interaction with an educational system. We follow \cite{yu2010feature} to extract features. 
\end{itemize}
The statistics of the data are summarized in Table
\ref{tbl:data}.\footnote{Downloaded from
	\url{http://www.csie.ntu.edu.tw/~cjlin/libsvmtools/datasets/}.}

We consider both linear SVM and L2-regularized logistic regression
discussed in Section \ref{subsec:binary}.
The comparison criteria are the relative primal and dual objective
distances to the optimum,
respectively defined as
\begin{equation}
	\left|\frac{f^P\left(\bw\left(\AL^t\right)\right) -
	f^*}{f^*}\right|,\quad
	\left|\frac{f(\AL^t) - (-f^*)}{f^*}\right|,
\label{eq:obj}
\end{equation}
where $f^*$ is the optimum  we obtained approximately by running our
algorithm with a tight stopping condition.
Note that the optimum for the dual and the primal problems are
identical except the flip of the sign, according to strong duality.
We examine the relation between these values and the training time.
We fix $C=1$ in this experiment.
The distributed environment is a cluster of $16$ machines connected
through MPI.

We compare the methods below whenever applicable.
\begin{itemize}
	\item \blockapprox: the Block-Diagonal Approximation method
		proposed in this paper.
		For the dual SVM problem, we utilize its quadratic objective
		to conduct exact line search while
		backtracking line search is used with the parameters being
		$\tau = 10^{-2}, \beta = 0.5$ for logistic regression.
		For $B_t$ in \eqref{eq:Ht}, we use $a_1^t \equiv
		1$ for all problems, as it is the closest block-diagonal
		approximation of the Hessian.  We set $a_2^t = 10^{-3}$
		in the hinge-loss SVM problem and $a_2^t = 0$ in the other two
		whose dual objectives are strongly convex.
	\item \disdca \citep{TY13b}: we use the practical
		variant for it outperforms th basic variant empirically.
		Moreover, experimental result in \cite{CM15b} showed that this
		algorithm (under a different name CoCoA+) is faster than
		\dsvm, and the best solver for the local
		sub-problems is indeed SDCA used in \cite{TY13b}.
	\item \lcommdir \citep{CPL17a}: a state-of-the-art
		distributed primal ERM solver that has been shown to
		empirically outperform existing distributed primal approaches.
We take the experimental setting in \cite{CPL17a} to use historical
information from the latest five iterations.
	\item \tron \citep{YZ14a,CYH17a}: a distributed implementation for
		ERM problems of the trust-region truncated Newton method
		proposed by \cite{TS83a}
\end{itemize}
All methods are implemented in C++.
We use the implementation of \lcommdir and \tron in the package \mpi
2.11.\footnote{\url{http://www.csie.ntu.edu.tw/~cjlin/libsvmtools/distributed-liblinear/}.}
These two methods require differentiability of the primal objective,
so we apply them only on squared-hinge loss SVM and logistic
regression problems.
We implement \disdca and \blockapprox with the local sub-problem solver
being the random permutation cyclic coordinate descent (RPCD) for dual
SVM \citep{CJH08a} and for dual logistic regression \citep{HFY10a}.
Note that the original solver in \cite{TY13b} is the dual stochastic
coordinate descent algorithm in \cite{SSS12a} that samples the
coordinates with replacement, but it has been shown in \cite{SSS12a}
that empirically RPCD is faster, and therefore we apply it in \disdca
as well.
At each iteration, we run one epoch of RPCD on each machine, namely
we pass through the whole dataset once, before communication.
This setting ensures that \disdca and \blockapprox have
computation-to-communication ratios similar to that of \lcommdir and
\tron,
so our results represent both a comparison for the training time
and a comparison for the number of communication rounds.

The comparison of the dual and primal objectives are shown in Figures
\ref{fig:webspam}-\ref{fig:kddb}.
For \webspam and \uu that are easier to solve, we present the result
of running different algorithms for $500$ seconds.
For the more difficult problem \kddb, we run all algorithms for
$10,000$ seconds.

We first discuss the dual objectives.
We can see that our approach is always better than state of the art
for the dual problem.
The difference is more significant in the SVM problems, showing that
low-cost exact line search has its advantage over backtracking,
while backtracking is still better than the fixed step size scheme.
The reason behind is that although the approach of \disdca provides a
safe upper bound model for the objective difference such that the
local updates can be directly applied to ensure the objective
decrease,
this upper bound might be too conservative as suggested by
\cite{CM15b}, but more aggressive upper bound modelings might be
computationally impractical to obtain.
On the other hand, our approach provides an efficient way to
dynamically estimate how aggressive the updates can be, depending to
the current iterate.
Therefore, the objective can decrease faster as the update is more
aggressive but still safe in terms of ensuring sufficient objective
value decrease.

Now we turn to the primal objectives.
Note that the step-like behavior of \blockapprox is from that we use
the best primal objective up to the current iterate discussed
in Section \ref{subsec:practical}.
Although aggressive step sizes in \blockapprox results in less stable
primal objective progress especially in the beginning, we observe that
\blockapprox still reaches lower primal objective faster than \disdca,
and the behavior of the early stage is less important.
For the case of hinge-loss SVM, \blockapprox is always the best, and
note that only dual approaches are feasible for hinge loss as it is
not differentiable.
When it comes to squared-hinge loss SVM, in which case exact line
search for the dual problem can still be conducted,
\blockapprox outperforms all primal and dual approaches.
The dual problem of logistic regression is not a quadratic one,
hence we cannot easily implement exact line search and need to resort
to the backtracking approach in Algorithm~\ref{alg:linesearch}.
We can see that for this problem, \lcommdir has an advantage in the
later stage of optimization,
while \blockapprox and \disdca are competitive till a medium
precision, which is usually enough for linear classification tasks.
In most cases, \tron is the slowest method.

\begin{figure*}
\begin{center}
	\begin{tabular}{@{}cc@{}}
Dual & Primal \\
\multicolumn{2}{c}{Hinge-loss SVM}\\
	\includegraphics[width=.45\textwidth]{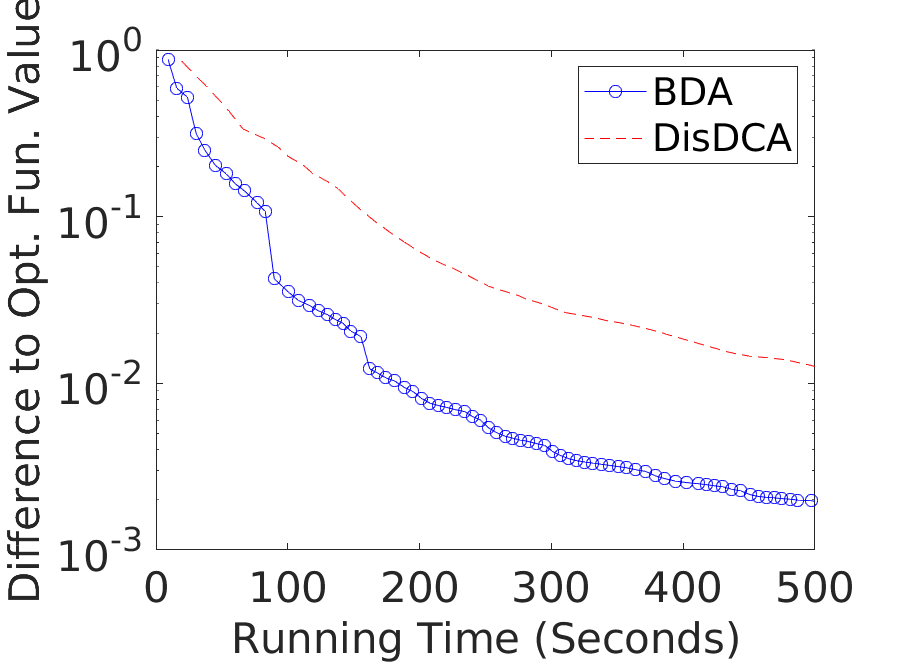} &
	\includegraphics[width=.45\textwidth]{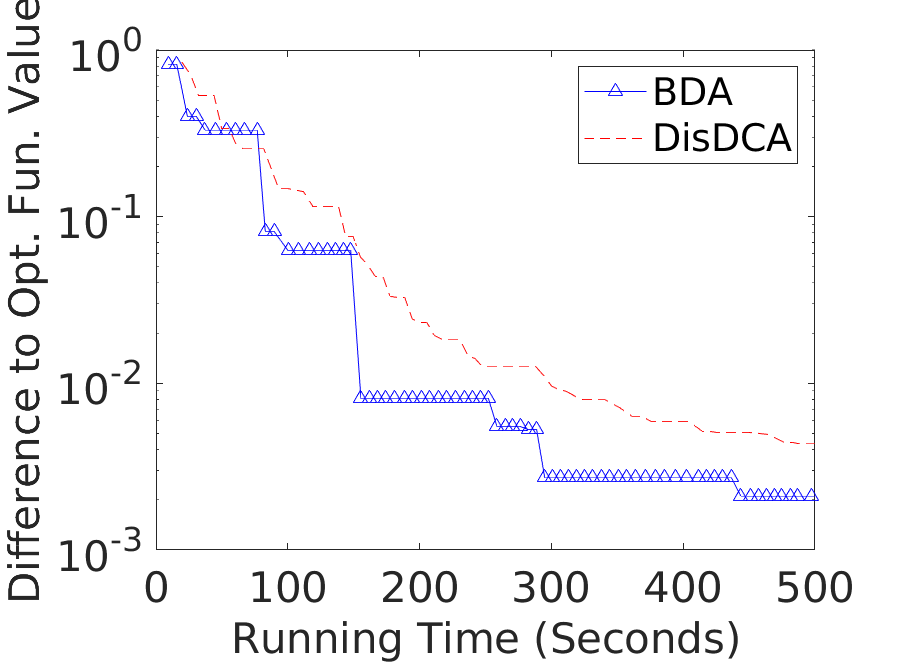}
	\\
\multicolumn{2}{c}{Squared-hinge loss SVM}\\
	\includegraphics[width=.45\textwidth]{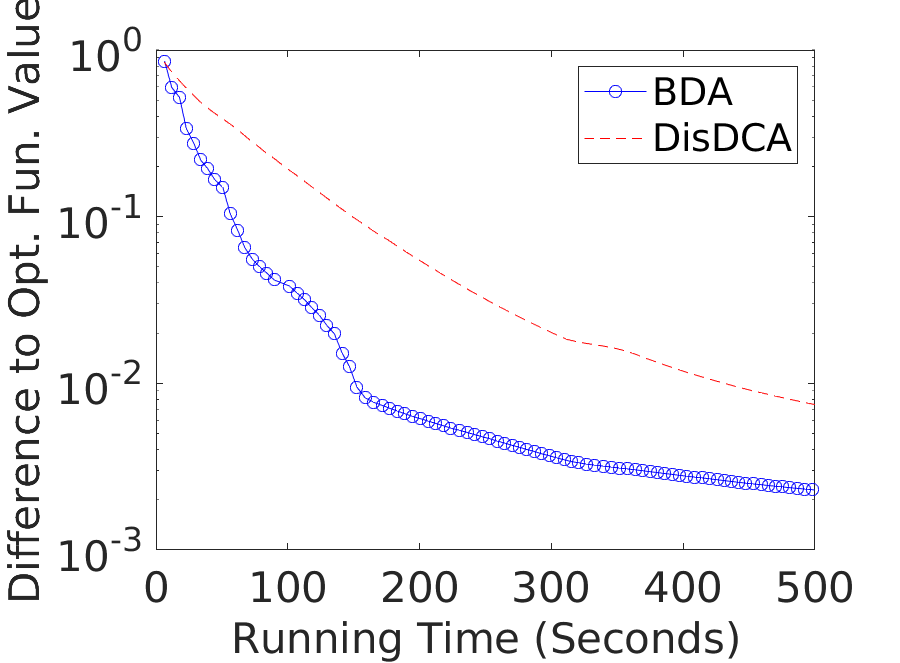} &
	\includegraphics[width=.45\textwidth]{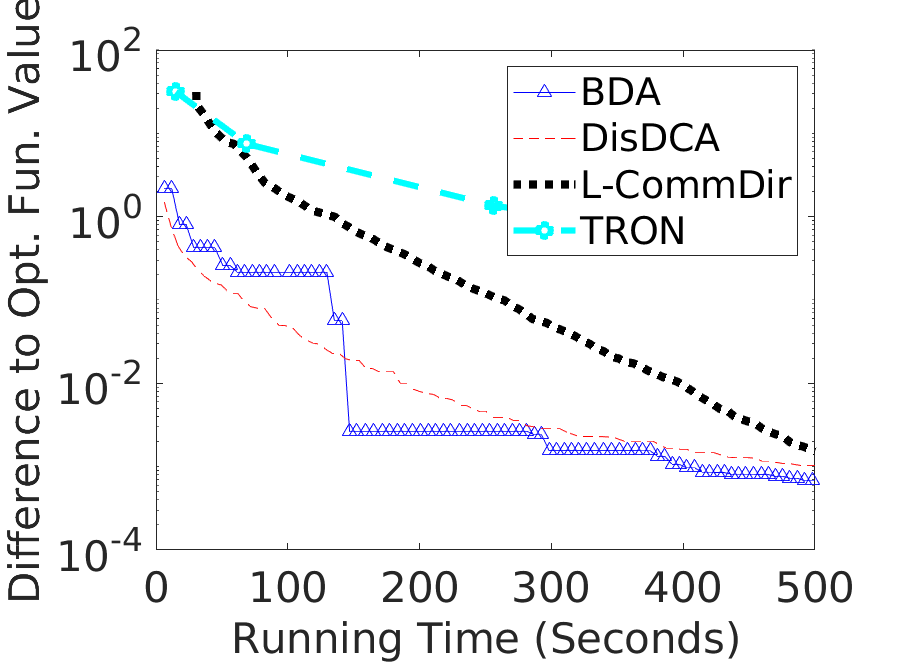}
	\\
\multicolumn{2}{c}{Logistic regression}\\
	\includegraphics[width=.45\textwidth]{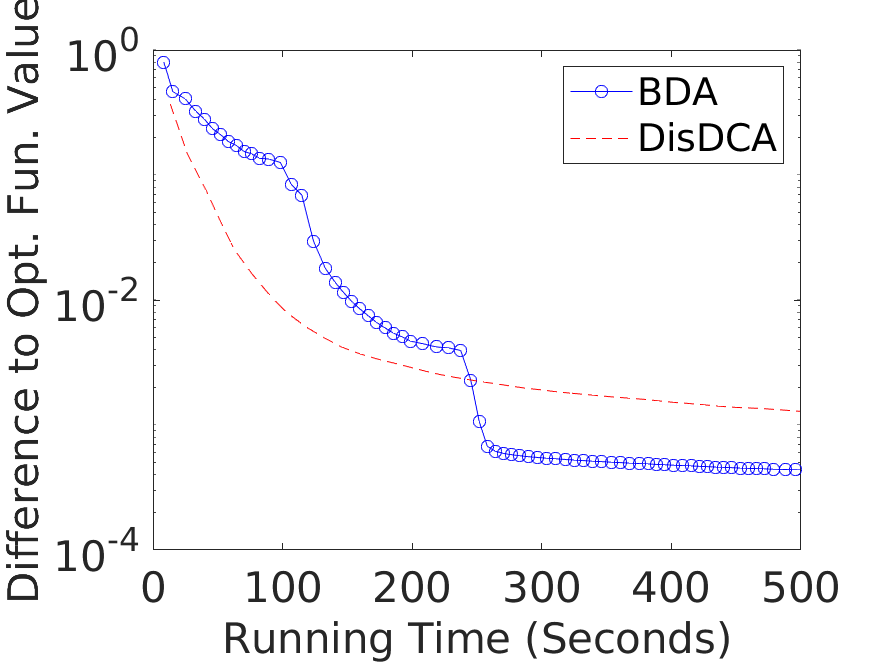} &
	\includegraphics[width=.45\textwidth]{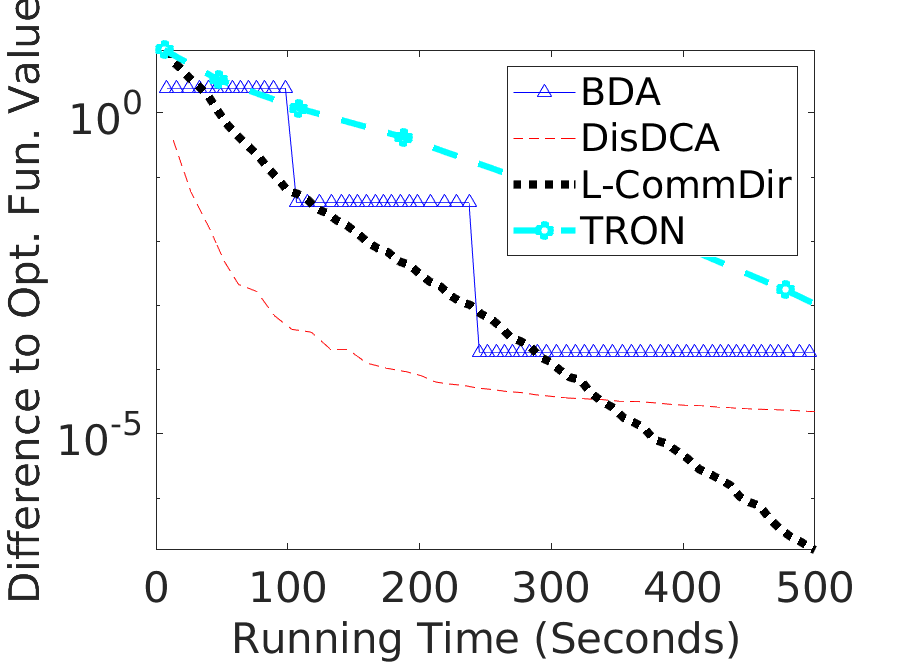}
\end{tabular}
\end{center}
\caption{Comparison of different algorithms for optimizing the ERM
problem on \webspam with  $C=1$. We show training time v.s. relative
difference of the objectives to the optimal function value.}
\label{fig:webspam}
\end{figure*}

\begin{figure*}
\begin{center}
	\begin{tabular}{@{}cc@{}}
Dual & Primal \\
\multicolumn{2}{c}{Hinge-loss SVM}\\
	\includegraphics[width=.45\textwidth]{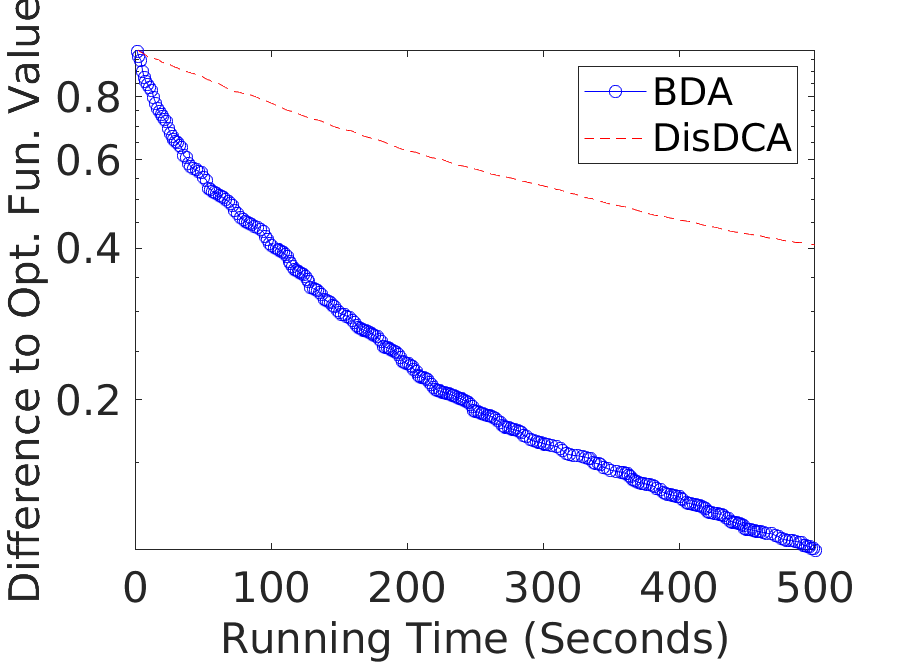} &
	\includegraphics[width=.45\textwidth]{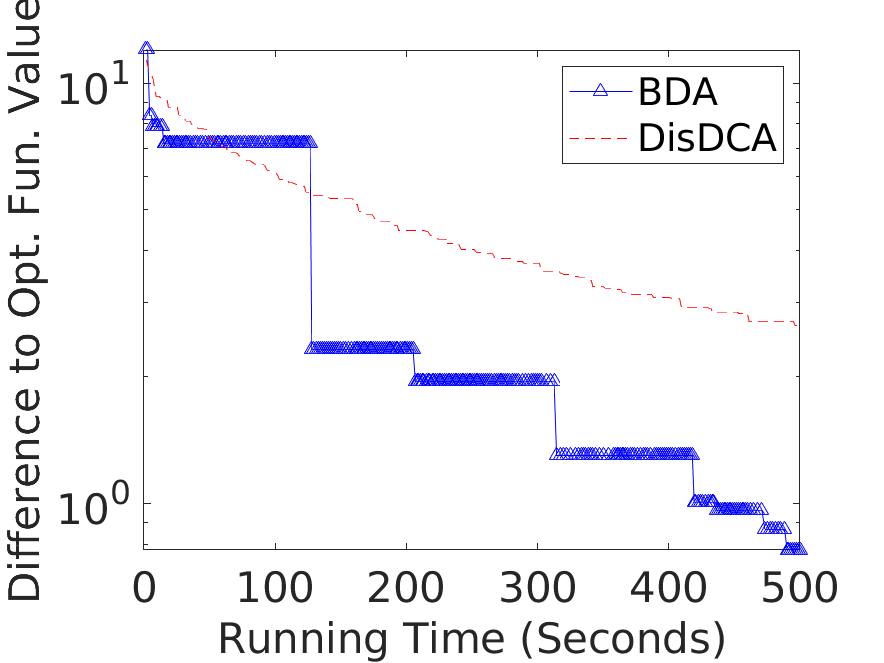}
	\\
\multicolumn{2}{c}{Squared-hinge loss SVM}\\
	\includegraphics[width=.45\textwidth]{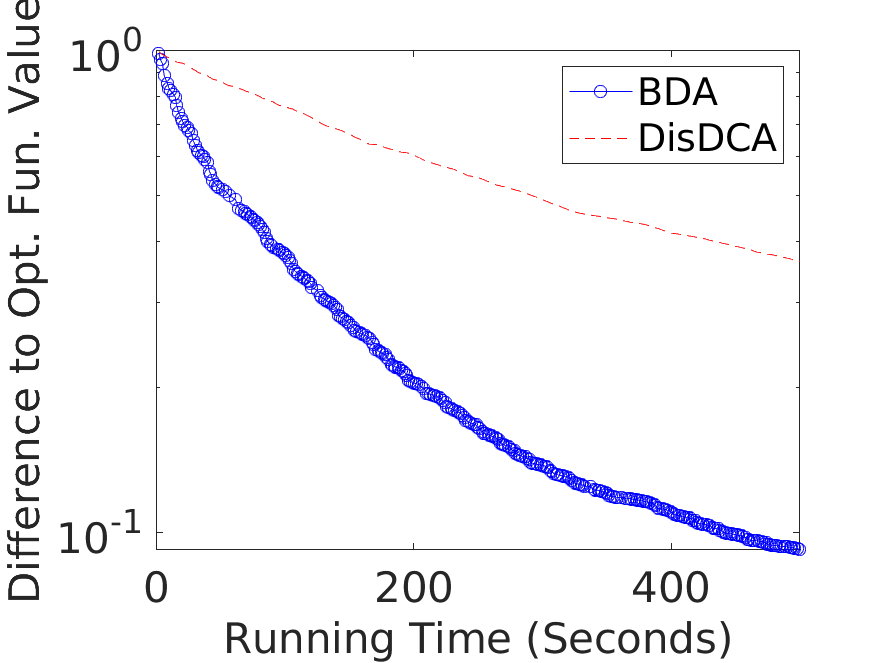} &
	\includegraphics[width=.45\textwidth]{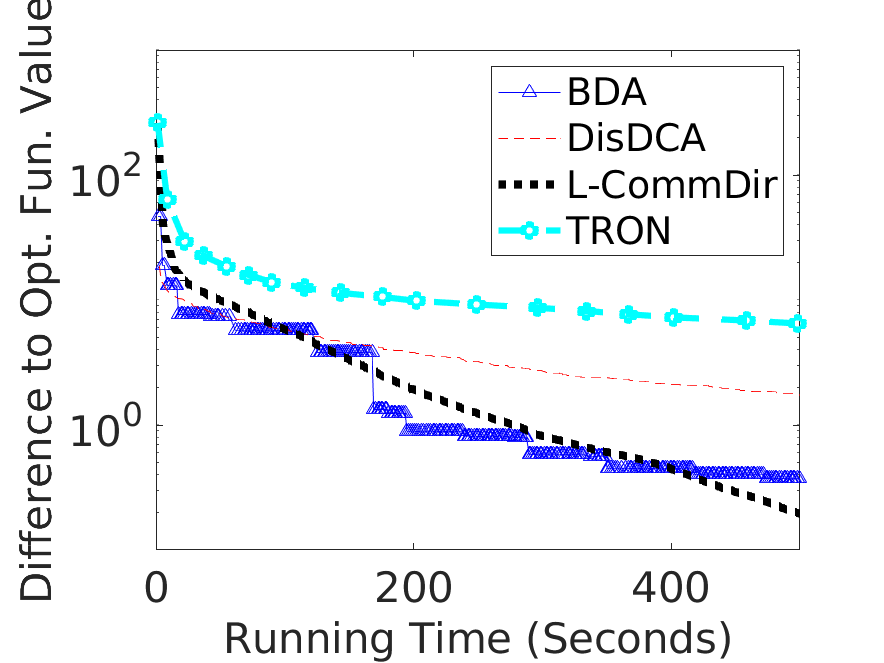}
	\\
\multicolumn{2}{c}{Logistic regression}\\
	\includegraphics[width=.45\textwidth]{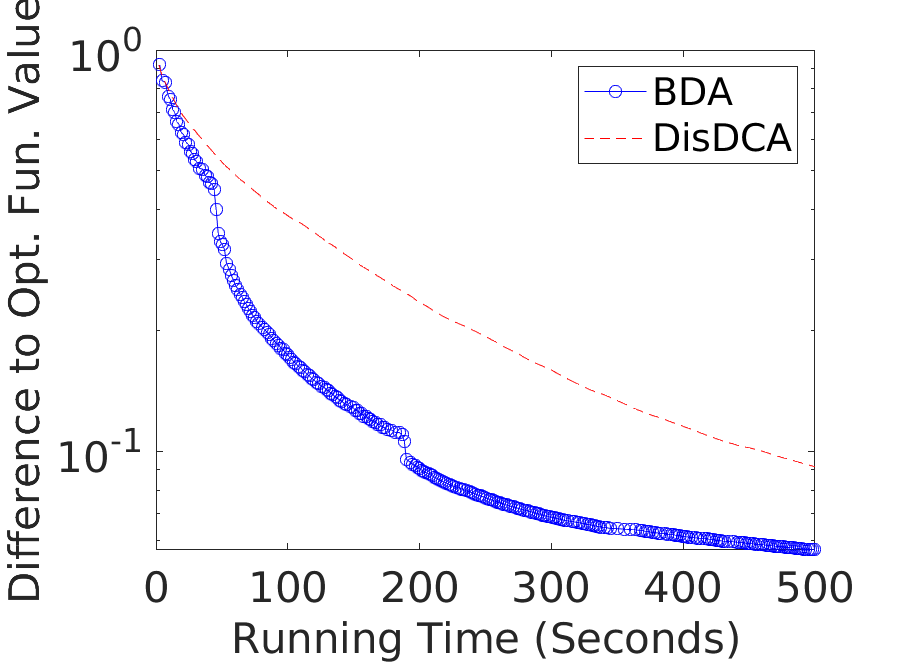} &
	\includegraphics[width=.45\textwidth]{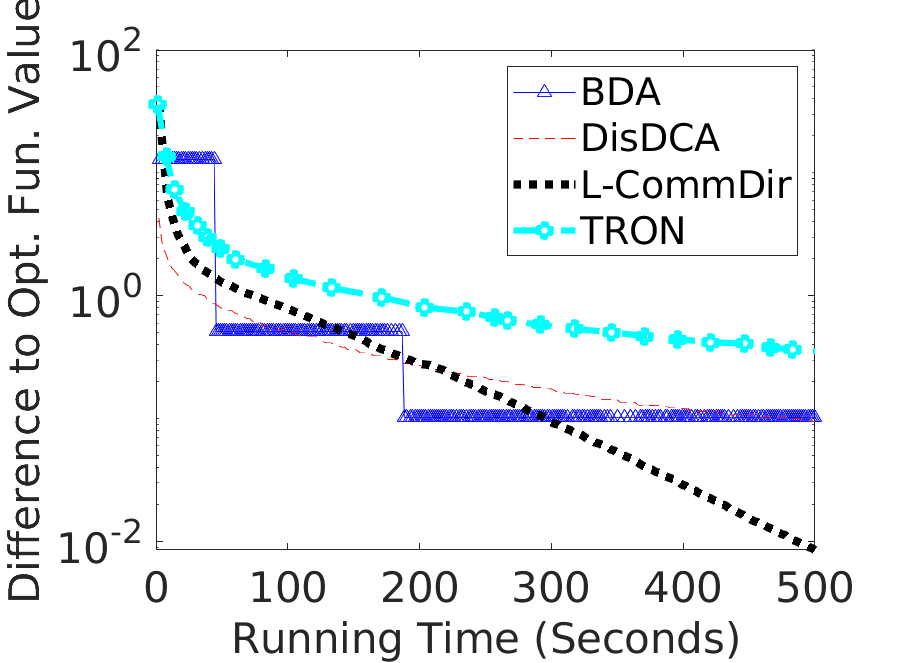}
\end{tabular}
\end{center}
\caption{Comparison of different algorithms for optimizing the ERM
problem on \uu with  $C=1$. We show training time v.s. relative
difference of the objectives to the optimal function value.}
\label{fig:url}
\end{figure*}

\begin{figure*}
\begin{center}
	\begin{tabular}{@{}cc@{}}
Dual & Primal \\
\multicolumn{2}{c}{Hinge-loss SVM}\\
	\includegraphics[width=.45\textwidth]{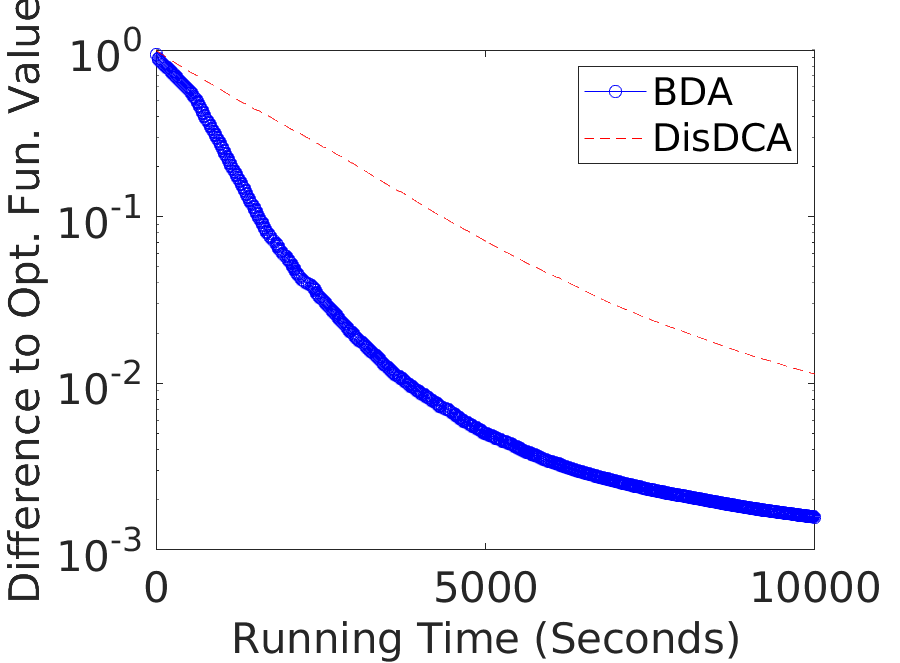} &
	\includegraphics[width=.45\textwidth]{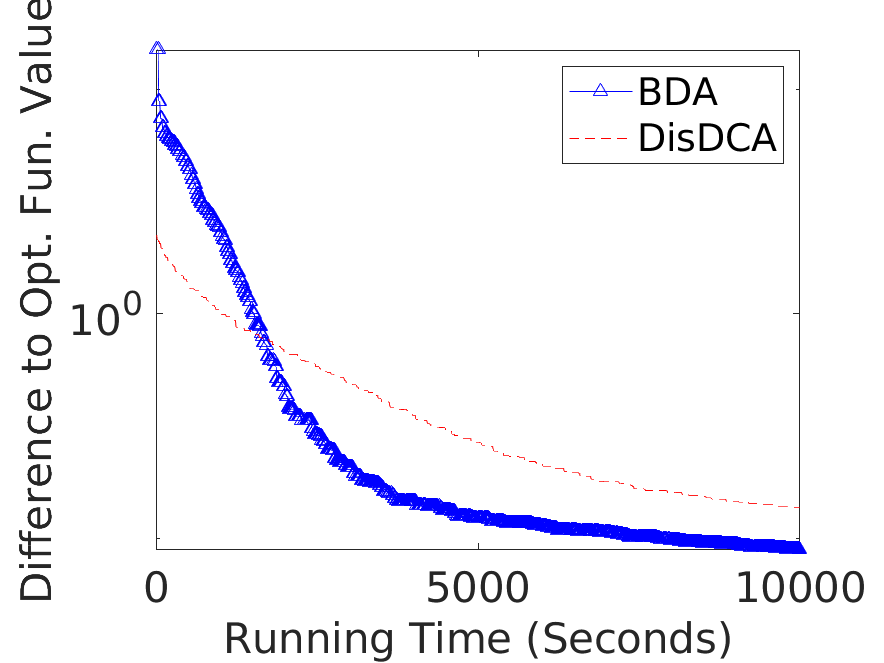}
	\\
\multicolumn{2}{c}{Squared-hinge loss SVM}\\
	\includegraphics[width=.45\textwidth]{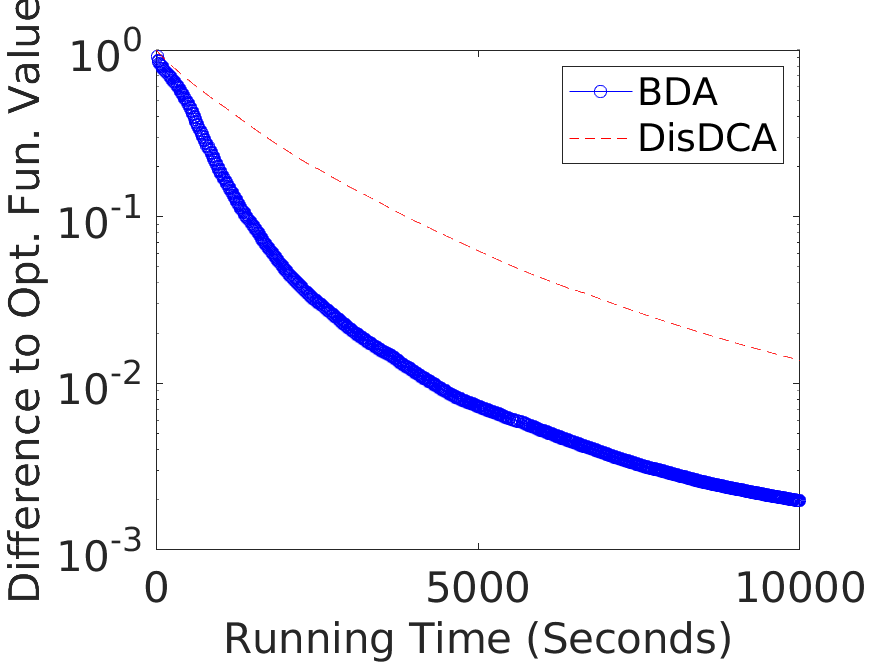} &
	\includegraphics[width=.45\textwidth]{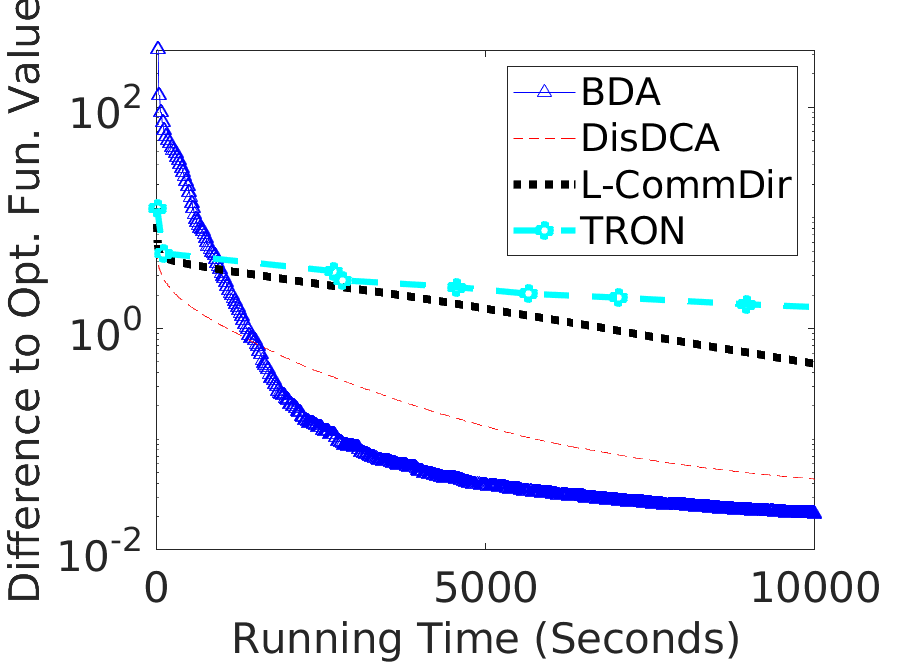}
	\\
\multicolumn{2}{c}{Logistic regression}\\
	\includegraphics[width=.45\textwidth]{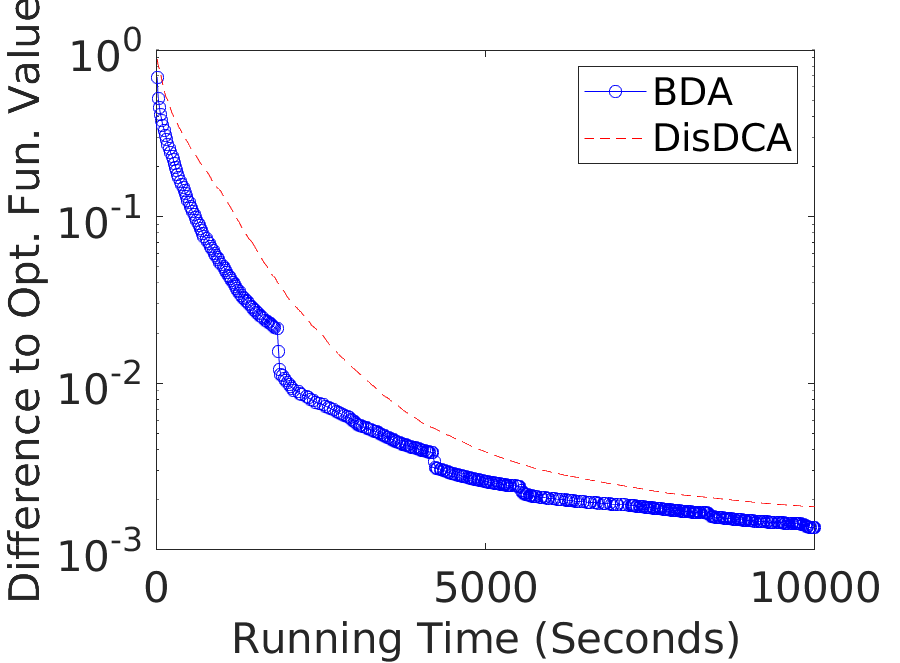} &
	\includegraphics[width=.45\textwidth]{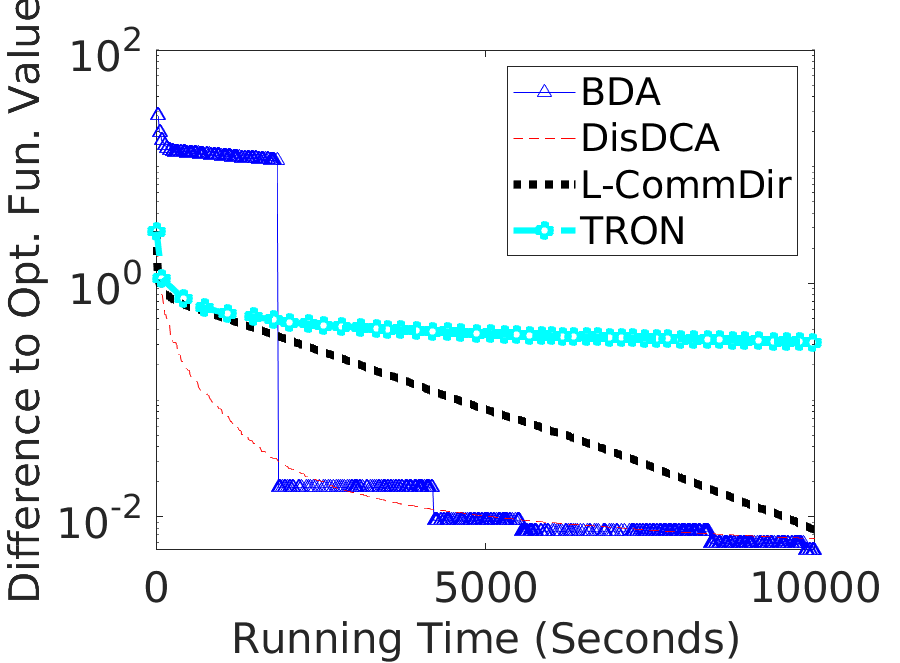}
\end{tabular}
\end{center}
\caption{Comparison of different algorithms for optimizing the ERM
problem on \kddb with  $C=1$. We show training time v.s. relative
difference of the objectives to the optimal function value.}
\label{fig:kddb}
\end{figure*}

\subsection{Structured Learning}
We perform experiments on two benchmark tasks for structured
prediction, part-of-speech tagging (\POS) and dependency parsing
(\dep).
For both tasks, we use the Wall Street Journal portion of the Penn
Treebank \citep{penn-tree-bank} with the standard split for training
(section 02-21), development (section 22), and test (section 23).
\POS is a sequential labeling  task, where we aim at learning
part-of-speech tags assigned to each word in a sentence.
Each tag assignment (there are 45 possible tag assignments) depends on
the associated word, the surrounding words, and their part-of-speech
tags. The inference in \POS is solved by the Viterbi algorithm
\citep{AV67a}.
We evaluate our model by the per-word tag accuracy.
For \dep, the goal is to learn, for each sentence, a tree structure
which describes the syntactic dependencies between words.
We use the graph-based parsing formulation and the features described
in \cite{MPRH05}, where we find the highest scoring parse using the
Chu-Liu-Edmonds algorithm \citep{ChuLiu65,JE67a}.
We evaluate the parsing accuracy using the unlabeled attachment score,
which measures the fraction of words that have correctly assigned
parents.

We compare the following algorithms using eight nodes in a local cluster.
All algorithms are implemented in JAVA,
and the distributed platform is MPI.
\begin{itemize}
	\item \blockapprox: the proposed algorithm. We take $a_1^t \equiv K$
		and $a_2^t \equiv 10^{-3}$ as $a_1^t \equiv 1$ is less stable
		in the primal objectives, which is essential for the
		sub-problem solver in this application.
	\item \ADMM: distributed alternating directions method of
		multiplier discussed in \cite{SB11b}.
	\item \perceptron: a parallel structured perceptron algorithm described in \cite{McDonaldHaMa10}.
	\item Simple average: each machine trains a separate
		model using the local data. The final model is obtained by averaging
		all local models.
\end{itemize}
For \blockapprox and \ADMM, the problem considered is SSVM in
\eqref{eq:ssvm} with L2 loss.
\perceptron, on the other hand, solves a similar but different problem
such that no regularization is involved.
We set $C=0.1$ for SSVM. Empirical experience suggests that structured SVM is not sensitive to $C$,
and the model trained with $C=0.1$ often attains reasonable test performance.

\begin{figure*}
	\centering
	\begin{tabular}{@{}c@{}c@{}}
\begin{subfigure}[b]{0.5\textwidth}
		\includegraphics[width=\textwidth]{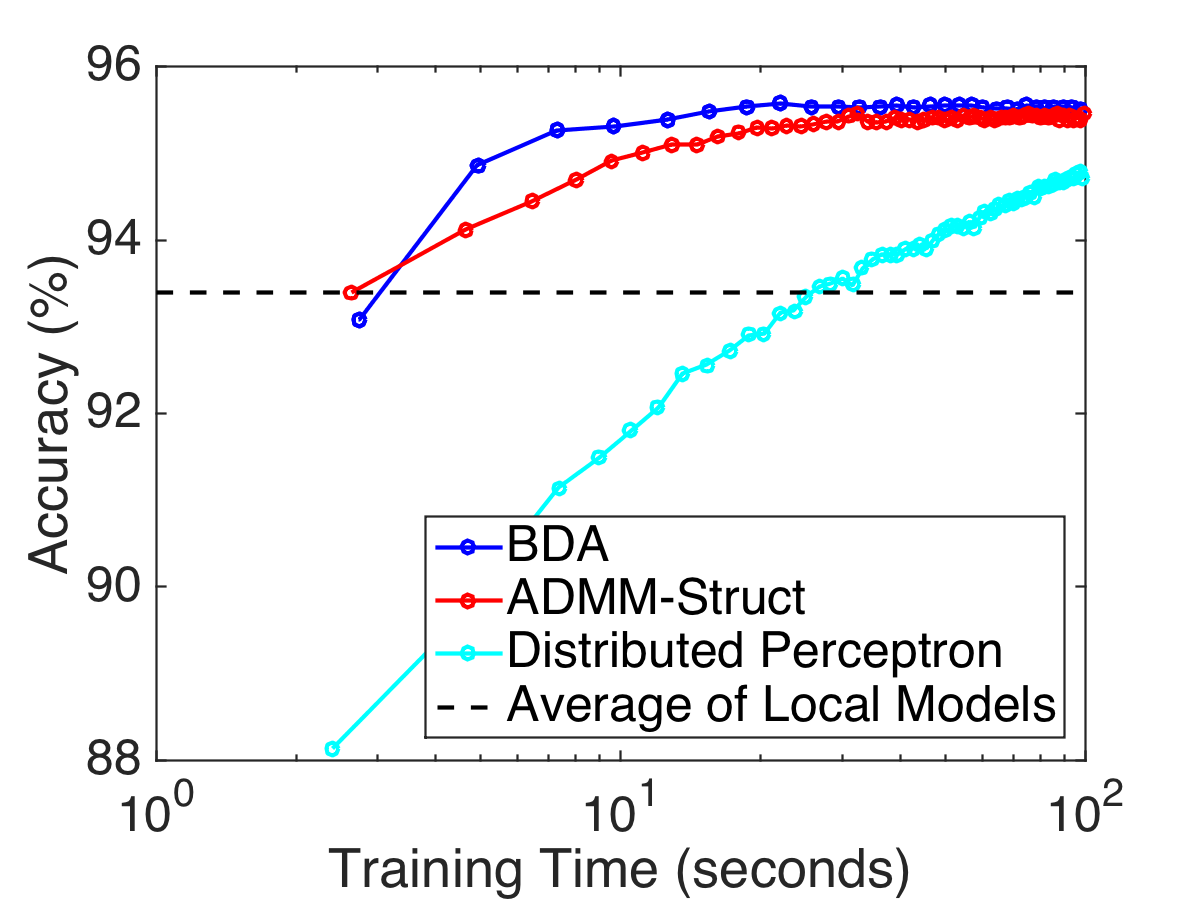} 
		\caption{\POS}
		\label{fig:pos}
	\end{subfigure}
		&
\begin{subfigure}[b]{0.5\textwidth}
		\includegraphics[width=\textwidth]{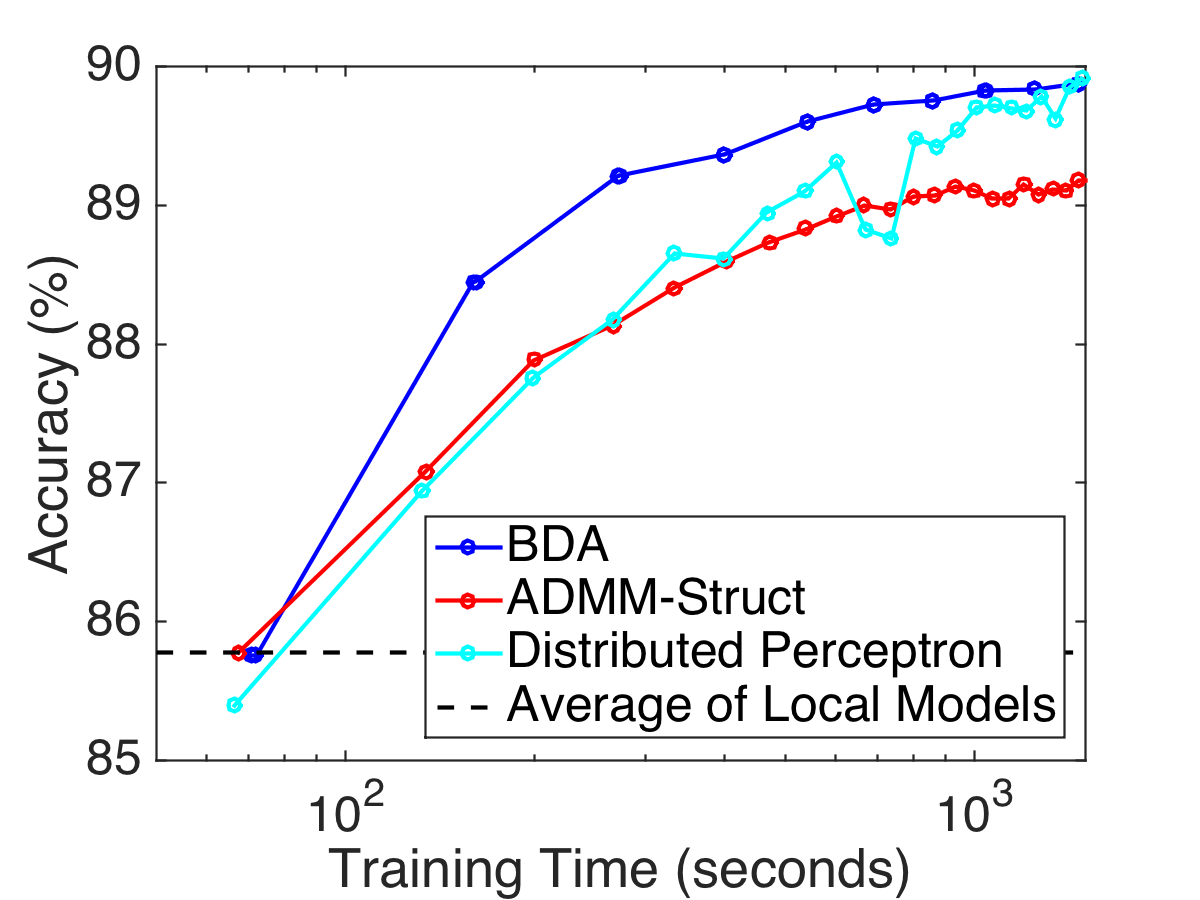} 
		\caption{\dep}
		\label{fig:dep}
	\end{subfigure}
\end{tabular}
\caption{Comparison between different algorithms for structured learning using eight nodes.
Training time is in {\it log scale}.}
	\label{fig:expstr}
\end{figure*}

Both \ADMM and \blockapprox decompose the original optimization problem into 
sub-problems, and we solve the sub-problems by the dual coordinate
descent solver for L2-loss SSVM proposed in \cite{ChangYi13}, which is
shown to be empirically efficient comparing to other existing methods.
By solving the sub-problems using the same optimizers, we can
investigate the algorithmic difference between \ADMM and \blockapprox.
For all algorithms, we fix the number of passes through all instances
to make inferences between any two rounds of communication to be
one, so that the number of inference rounds is identical to the number
of communication rounds. Although it is possible to alter the number
of inferences between two rounds of communication (or the number of
communication between two rounds of inferences) to obtain a faster
running time,
fine-tuning this parameter is not realistic for users because this
parameter does not affect the prediction performance, and thus there
is no reason to spend time retrain the model several times.
For \blockapprox and \ADMM,
each time in solving the local sub-problem with a fixed working set,
we let the local RPCD solver pass through the local instances ten times.
We note that this number of iterations may also affect the convergence
speed but we do not fine-tune this parameter for the same reason above.
For \ADMM, the weight for the penalty term in the augmented Lagrangian
also affects the convergence speed.\footnote{See, for example,
	\cite{SB11b} for details.}
Instead of fine-tuning it, a fixed value of $1.0$ is used.
Note that since \perceptron and \blockapprox/\ADMM consider different
problems,
instead of showing objective function values, 
we compare the test performance along training time of these methods.

Figure \ref{fig:expstr} shows the results. The x-axis is in \emph{log-scale}. 
Although averaging local classifiers achieves reasonable performance, 
all other methods improve the performance of the models with 
multiple rounds of communications.  This indicates that training models 
jointly on all parts of data is necessary. 
Among different algorithms, \blockapprox performs the best in both tasks.  
It achieves the final accuracy performance (indicated when the
accuracy stops improving) with shorter training time 
comparing to other approaches.  This result confirms that \blockapprox enjoys 
a fast convergence rate.  

\subsection{Line Search}
\label{subsec:lineexp}
The major difference between our approach and most other dual
distributed optimization methods for ERM is the line
search part.
In this subsection we examine the step sizes obtained in
practice and show that there is still a gap between
\eqref{eq:improved0} and \eqref{eq:improved} as the Lipschitz
parameter can be smaller in a local region.
We also investigate the empirical cost of line search.
For this investigation, we use the information from the L2-regularized
linear classification experiments in Section \ref{subsec:binaryexp}.

In Table~\ref{tbl:linesearchtime}, we show the proportion of time
spent on line search to the overall training time.
As the results indicate, the cost of line search is relatively low
in comparison to solving the local sub-problem and communicating
$\Delta \bv$.
Note that the cost of line search for hinge and squared-hinge loss is
independent to the final step size, as exact line search instead of
backtracking is applied.

\begin{table}
	\begin{center}
	\caption{Percentage of training time spent on line search.
		For hinge and squared-hinge loss, Variant II of
		Algorithm~\ref{alg:blockcd} is used, while Variant I is
	applied for logistic loss.}
	\label{tbl:linesearchtime}
	\begin{tabular}{l|rrr}
		Loss & \webspam & \uu& \kddb\\
		\hline
		Hinge & 6.25\% & 7.82\% & 2.05\% \\
		Squared-hinge & 9.87\% & 10.64\% & 7.50\% \\
		Logistic & 0.47\% & 5.00\% & 4.55\%
	\end{tabular}
	\end{center}
\end{table}

\subsection{Speedup}
Finally, we examine the practical speedup of \blockapprox.
We pick \webspam on L2-loss SVM as a representative example for this
experiment.
We run different algorithms on $\{1,2,4,8,16\}$ machines and see how
the training time and the overall running time (training time plus
data loading time) differ.
We record the time for \eqref{eq:obj} to reach $10^{-2}$ in Figure
\ref{fig:scale}.
The left column represents the time measured using the primal
objective, while the right column represents that using the dual
objective.
We can see that when it comes to the training time, \tron has a better
speedup because its algorithmic behavior is invariant of how the data
are distributed, while \blockapprox still enjoys better speedup than
the state of the art dual solver \disdca.

When the data loading time is combined, we can see that
\blockapprox has the best speedup, and the reason can be seen from the
third row of time profiling.
We see that the bottleneck in the single-machine case is data loading
which is embarrassingly parallel, and the training time of
\blockapprox is insignificant in comparison with the I/O time.
Therefore, although the training time speedup of \blockapprox is not
that significant, the running time speedup is very promising.
Another reason we cannot obtain good speedup in the training time is
that the single-machine case is already very efficient in comparison
with \tron, so it is rather difficult to have further improvement.

\begin{figure*}
\begin{center}
	\begin{tabular}{@{}cc@{}}
Dual & Primal \\
\multicolumn{2}{c}{Training Time Speedup}\\
\includegraphics[width=.45\textwidth]{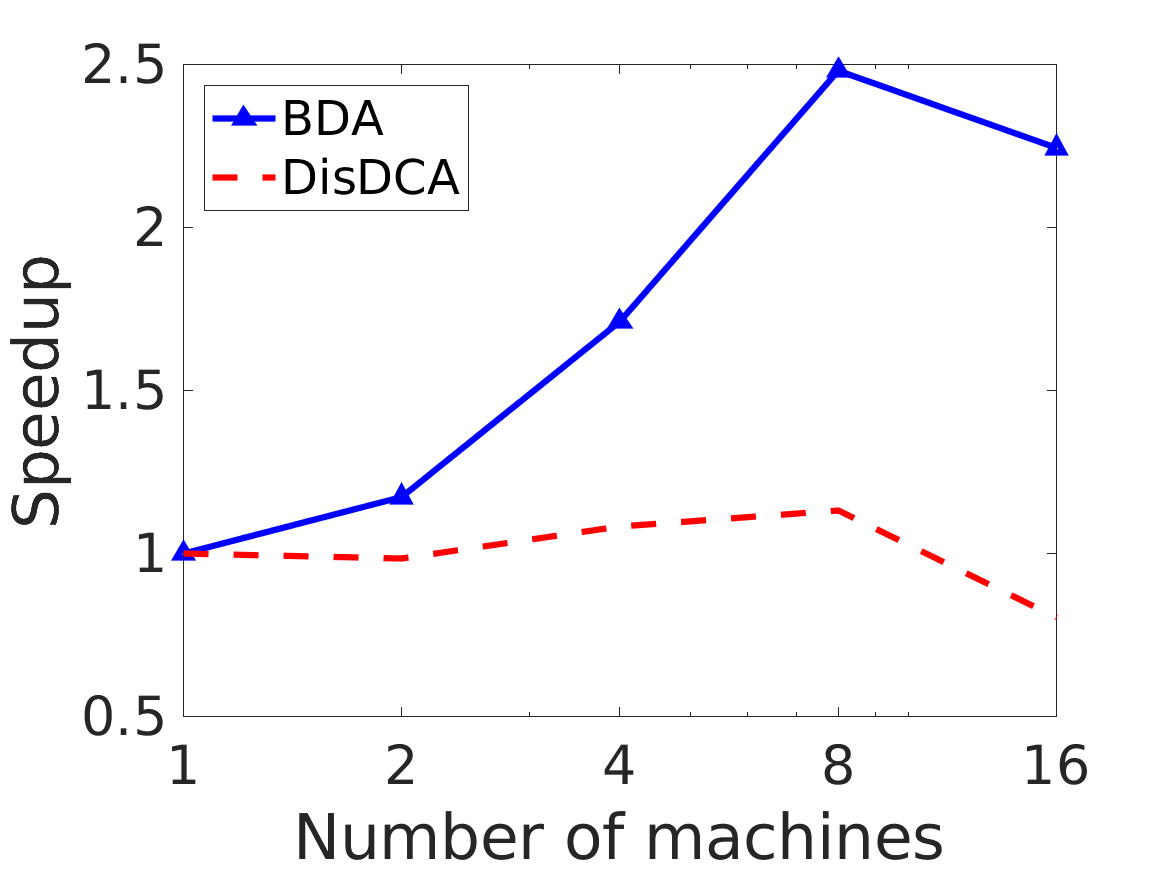} &
\includegraphics[width=.45\textwidth]{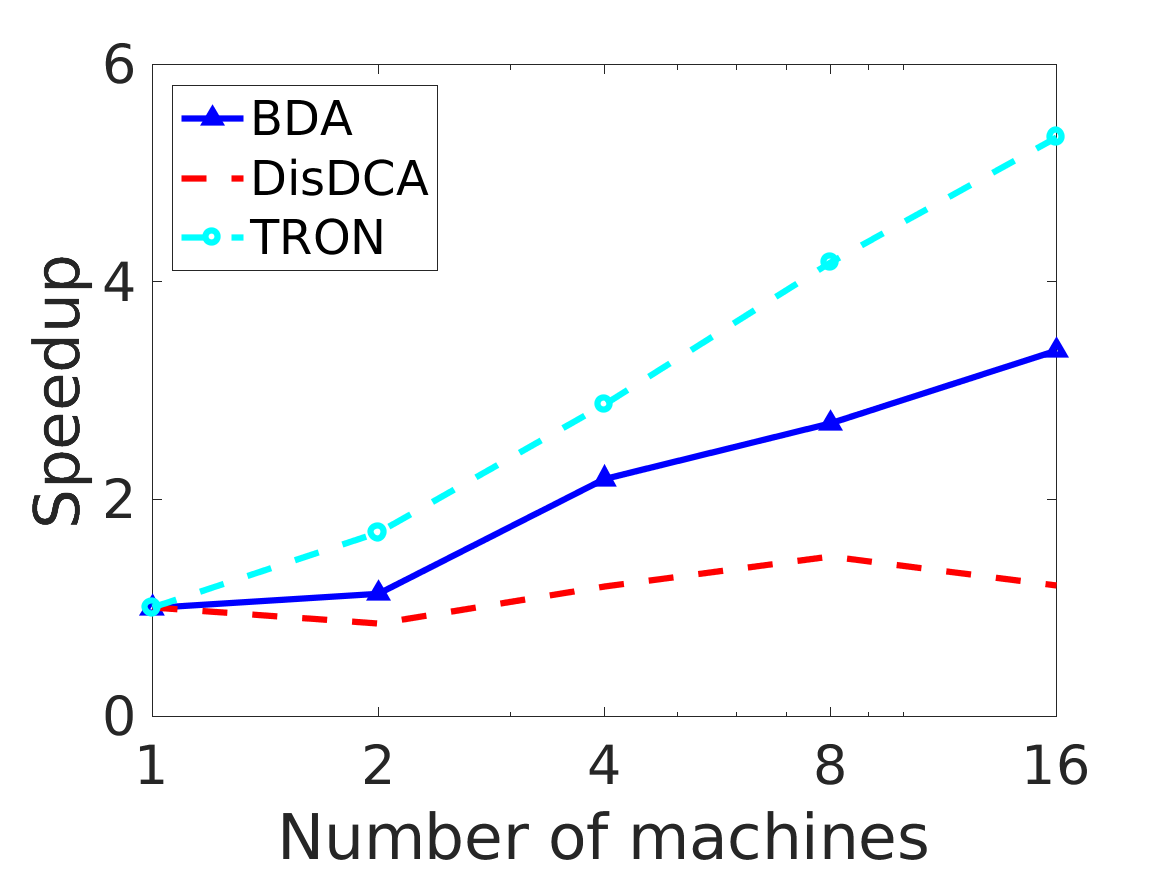}
	\\
\multicolumn{2}{c}{Overall Running Time Speedup}\\
\includegraphics[width=.45\textwidth]{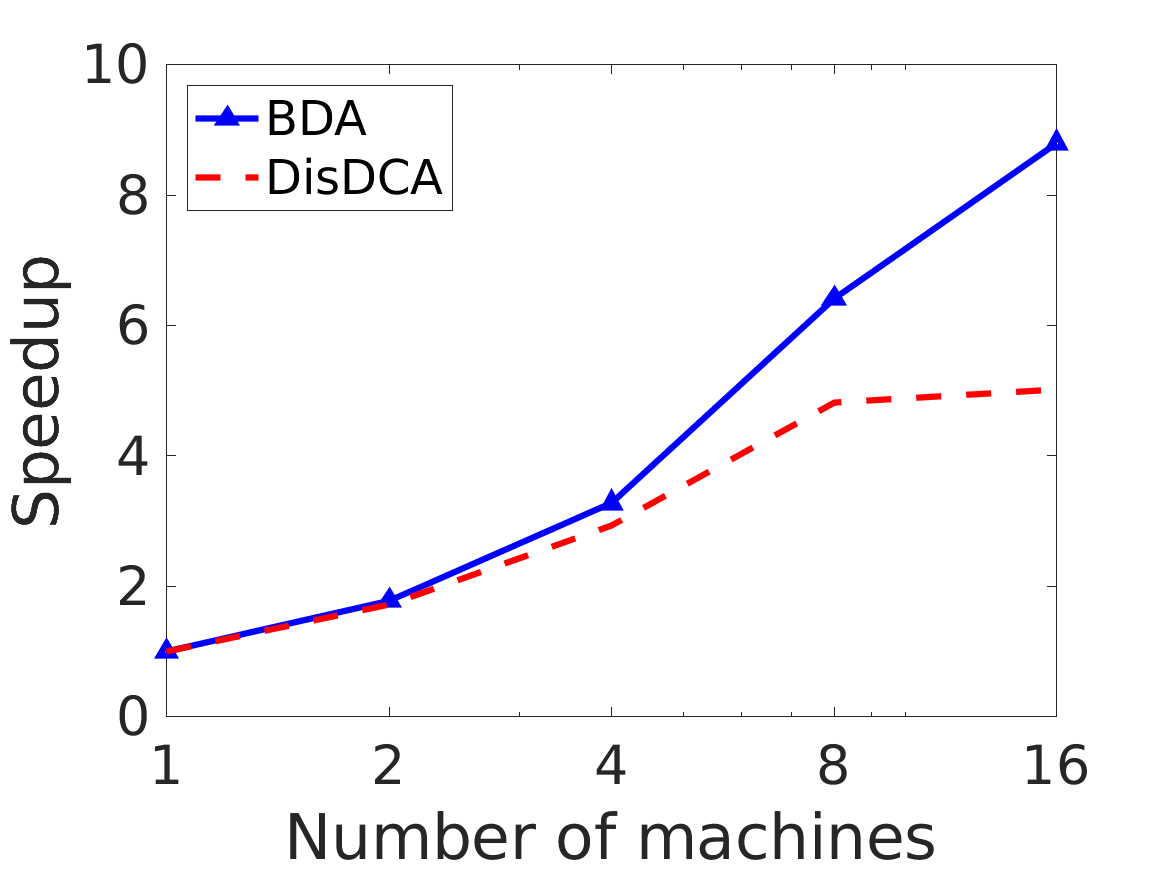}
&
\includegraphics[width=.45\textwidth]{figures/webspam.trdual_speedup.png}
	\\
\multicolumn{2}{c}{Scalability Factors}\\
\includegraphics[width=.45\textwidth]{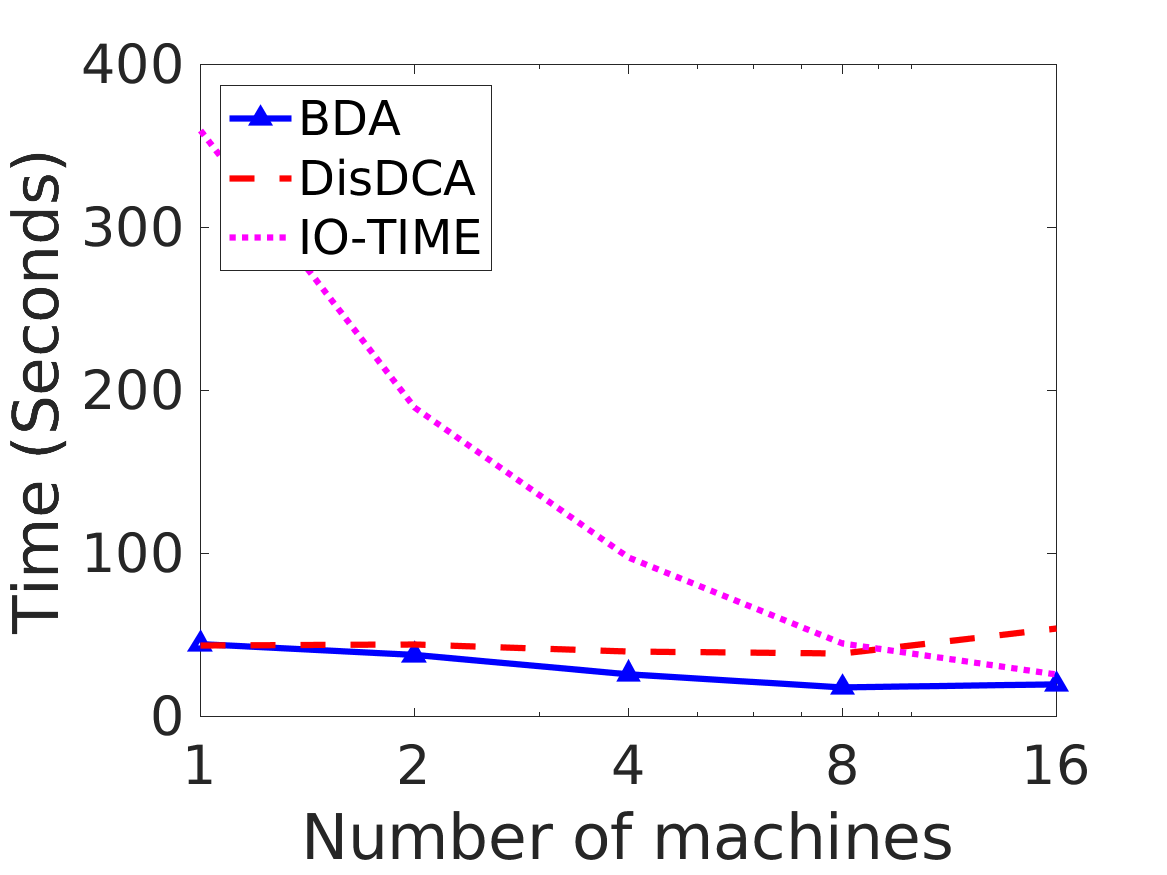}&
\includegraphics[width=.45\textwidth]{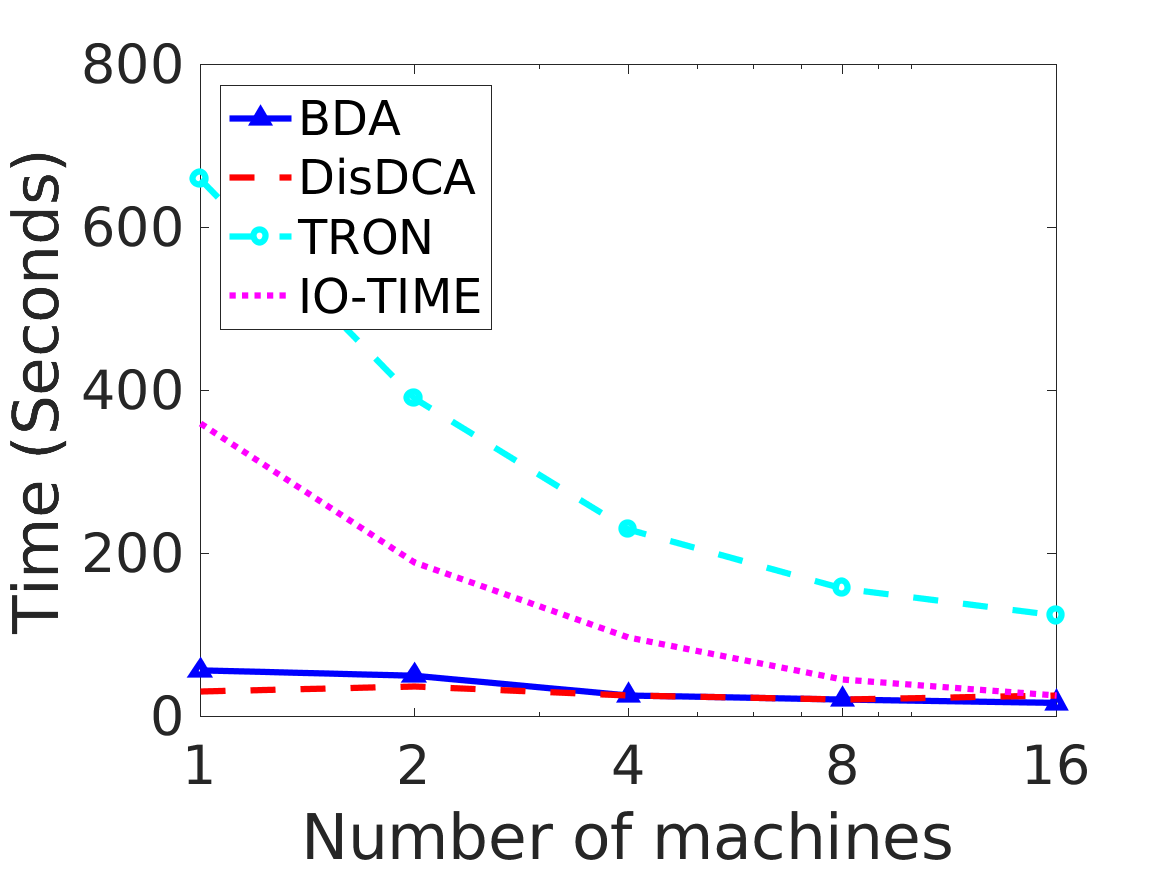}
\end{tabular}
\end{center}
\caption{Speedup of different algorithms for training L2-loss SVM
on \webspam with $C=1$.}
\label{fig:scale}
\end{figure*}

\section{Discussion}
\label{sec:discuss}
As Section~\ref{subsec:improved} suggests, if the block-diagonal
matrix $B_t$ is a tight approximation to the Hessian of $G^*$,
\blockapprox is expected to enjoy fast convergence.
To achieve so, we might partition the data in a better way such that
those off diagonal-block entries in the matrix $X^TX$  are as small as
possible, then the Hessian will also have smaller off-diagonal terms.
However, repartitioning the data across machines involves a significant
amount of data transmission, and designing an efficient mechanism to
split the data into blocks with desirable properties is challenging.
One practically feasible scenario is the case where the data points
are streamed in and partitioned in an online fashion.

Notice that in \eqref{eq:improved}, having a larger step size while
maintaining a large $\mu_B$ leads to fast convergence.
However, balancing these two factors is not an easy task.
One potential heuristic is to adjust $a_1^t$ and $a_2^t$ dynamically
based on the step size in the previous iteration.

One limitation of our current approach is that the algorithm does not
scale strongly with the number of machines when the data size is
fixed.
If the number of machines increases, $B_t$ will contain more zero entries.
This means the algorithm will be closer to a proximal gradient method and
converge slowly.
This is inevitable for all distributed dual optimizers we discussed in
Section~ \ref{sec:related}.
However, in many real applications, distributed optimization techniques
are used to protect privacy or handle distributional data. In such
applications, repartitioning data is costly and may not be feasible.
Therefore, the number of machines is predefined and  practitioners are
concerned more about how to make the optimization
procedure more efficient given the fixed number of machines and the fixed
data partitions, but not how to use more machines for the same data to
speed up the training process.

As mentioned in Section \ref{sec:related}, just like
\citet{ZS17a} applied existing acceleration techniques on top of
\disdca,
our algorithm can also be combined with the acceleration techniques
proposed by \citet{SSS13a,HL15a} to obtain a faster algorithm, and we
expect using our algorithm instead of \disdca will be faster than the
result in \cite{ZS17a} as our algorithm is faster than \disdca in
practice.
This comparison will be an interesting future work.

\section{Conclusions}
\label{sec:conclusion}
In this work, we proposed a distributed optimization framework for
the dual problem of regularized empirical risk minimization.
Our theoretical results show linear convergence for both the dual
problem and the corresponding primal problem for a variety class of
popular problems whose dual problem is non-strongly convex.
Our analysis further shows that when the sub-problem can serve as a
preconditioner to improve the problem condition, much better
convergence speed can be expected.
Our approach is most powerful when it is difficult to directly solve
the primal problem.
Experimental results show that our method outperforms state-of-the-art
distributed dual approaches for regularized empirical risk
minimization, and is competitive to cutting-edge distributed primal
methods when those primal methods are feasible.

\paragraph{Acknowledgement.}
The authors would like to thank the action editor and the anonymous
reviewers for their valuable comments and Dan Roth, Shyam Upadhyay,
Chih-Jen Lin, Cho-Jui Hsieh, Martin Jaggi, and Shai Shalev-Shwartz for
their feedback and suggestions on the early version of this paper.
This work was supported in part by NSF Grant IIS-1760523.
\bibliographystyle{spmpsci}     
\bibliography{local}

\appendix
\section{Proofs}
\subsection{Proof of Lemma \ref{lemma:strong}}
\begin{proof}
By \cite[Part E, Theorem 4.2.1]{HU01a},
if Assumption \ref{assum:LipGrad} holds, then $\bxi^*(\cdot)$
and hence $f$ is $(1/\rho)$-strongly convex.
We thus have that for any $\AL_1, \AL_2 \in \Omega$
and any $\lambda \in (0,1]$,
\begin{align*}
	f\left( \lambda\AL_1 + \left( 1-\lambda \right)\AL_2 \right)
	\leq \lambda f\left(\AL_1 \right) + \left(1-\lambda \right) f
	\left(\AL_2 \right) -
	\frac{\lambda(1-\lambda)}{2\rho}\left\|\AL_1 - \AL_2\right\|^2,
\end{align*}
which implies
\begin{align*}
f\left(\AL_1\right) - f\left(\AL_2\right)
\geq
	\frac{1 - \lambda}{2 \rho} \|\AL_1 - \AL_2\|^2 +
	\frac{f\left( \AL_2 + \lambda\left( \AL_1  - \AL_2
	\right)\right) - f\left( \AL_2 \right)}{\lambda}.
\end{align*}
Let $\lambda \rightarrow 0^+$, we get
\begin{align*}
f\left(\AL_1\right) - f\left(\AL_2\right)
\geq
	\frac{1}{2 \rho} \|\AL_1 - \AL_2\|^2 +
	\bs^T \left( \AL_1 - \AL_2 \right), \forall \bs \in \partial
	f\left( \AL_2 \right).
\end{align*}
By taking $\AL_2 = \AL$ with any $\bs \in \partial f(\AL)$ and
minimizing both sides with respect to $\AL_1$ simultaneously, we get
\eqref{eq:strong} as $\bs$ is arbitrary.

\end{proof}

\subsection{Proof of Lemma \ref{lemma:linesearch} and
\ref{lemma:linesearch2}}
We can see that Lemma \ref{lemma:linesearch} is a special case
of Lemma \ref{lemma:linesearch2} with $\gamma = 0$, so we
provide detailed proof for the latter only.

This result follows directly from \cite[Lemma~3]{CPL18a}, which
implies (in our notation)
\begin{equation}
	\Delta_t \leq - \frac12 \left( \frac{\left( 1 - \sqrt{\gamma}
\right)C_2 }{\left( 1 + \sqrt{\gamma} \right)} +
C_1\right)\left\|\DAL^t\right\|^2
	\label{eq:descent}
\end{equation}
and
\begin{equation*}
\eta_t \geq \min\left(1, \frac{\beta(1 - \tau) \sigma \left(\left( 1 +
	\sqrt{\gamma}\right)C_1 + \left( 1 - \sqrt{\gamma} \right)C_2
\right)}{\left\|X^T X\right\| \left( 1 + \sqrt{\gamma}
\right)}\right).
\end{equation*}

\subsection{Proof of Theorem \ref{thm:duallinear}}
\begin{proof}
	We first show the result for the variant of using backtracking
	line search.
	From \eqref{eq:update}, \eqref{eq:descent} with $\gamma = 0$, and
	\eqref{eq:armijo}, we have that
	\begin{equation}
		f(\AL^{t+1}) - f(\AL^t) \leq -\eta_t \tau
		\frac{C_1 + C_2 }{2}\|\DAL^t\|^2.
		\label{eq:improve}
	\end{equation}
	From the optimality of $\DAL^{t}$ in \eqref{eq:quadratic}, we
	get that
	\begin{equation}
		\nabla G^*(\AL^t) + B_t \DAL^t + \tilde\bs^{t+1} = 0,
		\label{eq:optimality}
	\end{equation}
	for some $\tilde\bs^{t+1} \in \partial \bxi^*(-\AL^t - \DAL^t)$.
	By convexity, that the step size is in $[0,1]$,
	and the condition \eqref{eq:strong}, we have
	\begin{align}
		f\left(\AL^{t+1}\right) - f^* &\leq \eta_t \left( f\left(\AL^t +
		\DAL^t\right) - f^*\right) + \left(1 - \eta_t\right)
		\left(f\left(\AL^t\right) - f^*\right) \nonumber\\
		&\leq \eta_t
		\frac{\|\nabla G\left(\AL^{t} + \DAL^t\right) +
		\tilde\bs^{t+1}\|^2}{2\mu} + \left(1 - \eta_t\right) \left(
		f\left(\AL^t\right) - f^*\right).
		\label{eq:intermediate}
	\end{align}
	Now to relate the first term to the decrease,
	we use \eqref{eq:optimality} to get
	\begin{align}
		\|\nabla G(\AL^{t} + \DAL^t) + \tilde\bs^{t+1}\|^2& \leq \|\nabla
		G^*(\AL^{t} + \DAL^t) -\nabla G^*(\AL^t) + \nabla G^*(\AL^t) +
		\tilde\bs^{t+1}\|^2\nonumber\\
		&\leq 2 \|\nabla G^*(\AL^{t} + \DAL^t) - \nabla G^*(\AL^t)\|^2 + 2
		\|B_t\DAL^t\|^2\nonumber\\
		&\leq 2 \left(\frac{\|X^T X\|}{\sigma}\right)^2 \|\DAL^t\|^2 + 2
		\|B_t\|^2 \|\DAL^t\|^2,
		\label{eq:intermediate2}
	\end{align}
	where in the second inequality, we used $(a+b)^2 \leq 2(a^2 +
	b^2)$ for all $a,b$, and in the last inequality we used the Lipschitz
	continuity of $\nabla G^*$.
	We therefore get the following by combining
	\eqref{eq:intermediate}, \eqref{eq:intermediate2},
	and \eqref{eq:improve}.
	\begin{align}
		\nonumber
		f\left( \AL^{t+1} \right) - f^*
		\leq &~\frac{\eta_t}{\mu}
		\left(\left(\frac{\|X^T X\|}{\sigma}\right)^2 + C_3^2 \right) \|\DAL^t\|^2 +
		\left(1 - \eta_t \right) \left( f\left(\AL^t \right) - f^*
		\right) \nonumber\\
		\leq &~\left( \left(\frac{\|X^T X\|}{\sigma}\right)^2 + C_3^2\right)
		\frac{2 \left( f \left(\AL^t \right) -f \left(\AL^{t+1}
	\right) \right) }{\mu(C_1 + C_2 )\tau } + \left(1 - \eta_t \right)
		\left(f \left(\AL^t \right) - f^*\right).
		\label{eq:bound2}
	\end{align}
	Let us define
	\begin{equation*}
		C_4 \coloneqq \left(\left(\frac{\|X^T X\|}{\sigma}\right)^2 + C_3^2
		\right)\frac{2}{\mu \left( C_1 +
			C_2 \right)
		\tau},
	\end{equation*}
	then rearranging \eqref{eq:bound2} gives
	\begin{equation}
		(f(\AL^{t+1}) - f^*) \leq \frac{(1 - \eta_t + C_4)}{1+C_4} (f(\AL^t) -
		f^*).
		\label{eq:qlinear}
	\end{equation}
	Combining the above result with the lower bound of $\eta_t$
	from \eqref{eq:stepsize}
	shows the desired $Q$-linear convergence rate.
	As of the exact line search variant,
	it produces an objective no larger than
	the left-hand side of \eqref{eq:qlinear}, so the same rate holds.
\end{proof}

\subsection{Proof of Theorem \ref{thm:duallinear2}}
\begin{proof}
This a direct application of \cite[Theorem 1]{WP18a}.
Their result implies
\begin{equation}
\frac{f\left( \AL^{t+1} \right) - f^*}{f \left( \AL^t \right) -
f^*}
\leq
\begin{cases}
1 - \eta_t \tau \left(1 - \gamma \right)
\frac{\mu}{4C_3}, &\text{ if } \mu \le 2 C_3,\\
1 - \eta_t \tau \left(1 - \gamma \right)
\left( 1 - \frac{C_3}{\mu} \right), &\text{else}.
\end{cases}
\label{eq:oneiter}
\end{equation}
Using \eqref{eq:stepsize} in \eqref{eq:oneiter}, we obtain
\eqref{eq:linear}.
\end{proof}

\subsection{Proof of Theorem \ref{thm:dualitygap}}
\begin{proof}
	Our proof consists of using $\AL$ as the initial point,
	applying one step of some primal-dual algorithm, then utilizing
	the algorithm-specific relation between the decrease in one
	iteration and the duality gap to obtain the bound.
	Therefore we will obtain an algorithm-independent result from some
	algorithm-specific results.

	When Assumption \ref{assum:LipGrad} holds, \eqref{eq:primal}
	is the type of problems considered in \cite{SSS12b},
	and we have that $\bxi^*$ is $(1/\rho)$-strongly convex.
	If we take $\AL$ as the initial point, and apply one
	step of their method to obtain the next iterate $\AL^+$,
	from \cite[Lemma 1]{SSS12b},
	we get that for any $s \in [0,1]$,
	\begin{align}
		\epsilon &= f\left(\AL\right) - f\left(\AL^*\right)
		\geq f\left(\AL\right) - f\left(\AL^+ \right)
		\geq s \left(f^P\left(\bw(\AL)\right) + f\left(\AL\right)\right) -
		\frac{s^2 G_s}{2 \sigma}\nonumber\\
		&\geq s \left(f^P\left(\bw(\AL)\right) -
		f^P\left(\bw^*\right)\right) - \frac{s^2 G_s}{2
		\sigma},
		\label{eq:Gradgap}
	\end{align}
	where
	$\bw^*$ is the optimal solution of \eqref{eq:primal}, and
	\begin{equation*}
		G_s \coloneqq \left(\|X^TX\| - \frac{\sigma (1 -
		s)}{s\rho}\right)\left\|\bu - \AL\right\|^2,\quad -\bu_i \in
		\partial \xi_i\left(X_i^T \bw\left( \AL \right) \right).
	\end{equation*}
	To remove the second term in \eqref{eq:Gradgap}, we set
	\begin{equation*}
		\|X^T X\| - \frac{\sigma(1 - s)}{s\rho } = 0
		\quad \Rightarrow \quad s = \frac{\sigma}{\sigma+ \rho \|X^T
		X\|} \in [0,1].
	\end{equation*}
	This then gives
	\begin{equation*}
	\left(1 + \frac{\rho\|X^T X\|}{\sigma}\right) \epsilon \geq
	f^P\left(\bw(\AL)\right) -
	f^P\left(\bw^*\right).
	\end{equation*}
	Although \cite[Lemma 1]{SSS12b} is for the expected value of the
	dual objective decrease at the current iteration and the expected
	duality gap at the previous iteration, we can remove the
	expectations as the expected duality gap is actually a constant
	for the initial point, and the expected function decrease cannot
	exceed the distance from the current objective to the optimum.

	When Assumption \ref{assum:Lipsloss} holds, \eqref{eq:primal}
	falls in the type of problems discussed in \cite{FB15a}.
	If we take $\AL$ as the initial point, and apply one
	step of their method to obtain the next iterate $\AL^+$,
	from the final inequality in the proof of Proposition 4.2 in
	\cite{FB15a} and weak duality, we get
	\begin{equation}
		\label{eq:Lipsbound}
		\epsilon
		\geq
		s \left(f^P\left(\bw(\AL)\right) -
		f^P\left(\bw^*\right)\right) - \frac{\left(sR\right)^2}{2
		\sigma}, \quad \forall s \in [0,1],
	\end{equation}
	where
	\begin{equation}
		\label{eq:R}
		R^2 \coloneqq \max_{\AL, \BL \in \Omega} \left\|X \left(\AL -
		\BL \right)\right\|^2
		\leq \left\|X^T X\right\| \max_{\AL, \BL \in \Omega}
		\left\|\AL -
		\BL\right\|^2
		\leq 4\left\|X^T X\right\| L^2.
	\end{equation}
	In the last equality we used \cite[Corollary 13.3.3]{RTR70a} such
	that if $\phi(\cdot)$ is $L$-Lipschitz continuous, then the
	radius of $\text{dom}(\phi^*)$ is no larger than $L$.
	Now take $s = \min\{1, \sqrt{2 \sigma \epsilon / R^2}\}$,
	we get that
	\begin{equation*}
		\begin{cases}
			2 \epsilon \geq \epsilon + \frac{R^2}{2 \sigma} \geq
			f^P\left(\bw\right) -
			f^P\left(\bw^*\right),
			& \text{ if } \epsilon \geq \frac{R^2}{2\sigma},\\
			\sqrt{\frac{2R^2\epsilon}{\sigma}} \geq f^P\left(\bw\right) -
		f^P\left(\bw^*\right), &\text{ else}.
		\end{cases}
	\end{equation*}
	These conditions and \eqref{eq:R} indicate that
	\begin{equation*}
		f^P\left(\bw\right) - f^P\left(\bw^*\right) \leq \max\left\{ 2
			\epsilon,
		\sqrt{\frac{2\epsilon R^2 }{\sigma}}\right\}
		\leq \max\left\{2 \epsilon,
		\sqrt{\frac{8\epsilon \|X^T X\| L^2}{\sigma}}\right\}.
	\end{equation*}
\end{proof}

\subsection{Proof of Lemma~\ref{lemma:improvedline}}
\begin{proof}
The part of \eqref{eq:deltabdd} follows directly from the proof of
\cite[Corollary~1]{CPL18a}.
Notice that they required positive definiteness of the matrix for other
parts stated in that corollary but the part for \eqref{eq:deltabdd}
holds true as long as $B_t$ is positive semidefinite.

For the lower bound on the step size, we notice that for any $\eta \in
[0,1]$,
\begin{align}
\nonumber
&~ f\left( \AL^t + \eta \Delta\AL^t \right) - f\left( \AL^t \right)\\
\nonumber
= &~ G^*\left(\AL^t + \eta \Delta\AL^t   \right) - G^*\left( \AL
	\right) + \bxi^*\left(-\AL^t - \eta \Delta\AL^t\right) -
	\bxi^*\left( -\AL^t \right)\\
\label{eq:tmp1}
\le &~ \eta \nabla G^* \left( \AL^t \right)^T \Delta\AL^t + \frac{\eta^2
	L_B}{2} \left\|\Delta\AL^t\right\|^2_{B_t} + \bxi^* \left( -\AL^t
	- \eta \Delta\AL^t\right) - \bxi^*\left( -\AL^t \right)\\
\label{eq:tmp2}
\le &~ \eta \nabla G^* \left( \AL^t \right)^T \Delta\AL^t + \eta
	\left( \bxi^* \left( -\AL^t - \Delta\AL^t \right) - \bxi^* \left(
	-\AL^t \right)\right) + \frac{\eta^2 L_B}{2} \left\| \Delta\AL^t
	\right\|^2_{B_t}\\
= &~ \eta \Delta_t + \frac{L_B \eta^2}{2} \left\| \Delta \AL^t
	\right\|^2_{B_t}
\label{eq:tmp3}
\le \eta \Delta_t - \frac{L_B \eta^2 \left( 1 + \sqrt{\gamma}
	\right)}{2} \Delta_t,
\end{align}
where we used \eqref{eq:betterLip} in \eqref{eq:tmp1}, the convexity
of $\bxi^*(-\cdot)$ in \eqref{eq:tmp2}, and \eqref{eq:deltabdd} in
\eqref{eq:tmp3}.
We can therefore see that \eqref{eq:armijo} is satisfied when
\begin{equation*}
	\left( \eta - \frac{L_B \eta^2 \left( 1 + \sqrt{\gamma}\right)}{2}
	\right) \Delta_t \le \eta \tau \Delta_t.
\end{equation*}
As $\Delta_t \le 0$ from \eqref{eq:deltabdd}, we have that
\eqref{eq:armijo} holds whenever
\begin{equation*}
\eta \le \frac{2 \left( 1 - \tau \right)}{L_B \left( 1 + \sqrt{\gamma}
\right)}.
\end{equation*}
After considering the overshoot of backtracking by a factor of
$\beta$, this inequality leads to the desired step size bound.
\end{proof}

\subsection{Proof of Theorem~\ref{thm:improved}}
\begin{proof}
We define $(\Delta \AL^t)^*$ as the optimal solution for
\eqref{eq:quadratic} at the $t$-th iteration and start from
\cite[Lemma~5]{CPL18a}, which states that when $f$ is convex, we have
the following inequality.
\begin{equation*}
	Q^{\AL^t}_{B_t} \left( \left( \Delta\AL^t \right)^*\right) \leq
	-\lambda \left( f\left( \AL^t \right) - f^* \right) +
	\frac{\lambda^2}{2}\min_{\AL^* \in A}\left\| \AL^t - \AL^*
	\right\|^2_{B_t}, \forall \lambda \in [0,1].
\end{equation*}
By utilizing \eqref{eq:betterQG}, we further deduce that
\begin{equation}
	Q^{\AL^t}_{B_t} \left( \left( \Delta\AL^t \right)^*\right) \leq
	-\lambda \left( f\left( \AL^t \right) - f^* \right) +
	\frac{\lambda^2}{\mu_B} \left( f\left( \AL^t \right) - f^*
	\right),\forall \lambda \in [0,1],
\label{eq:bound_improve}
\end{equation}
which attains the minimal value at $\lambda = \mu_B / 2$, which is in
$[0,1]$ according to our assumption that $\mu_B \leq 2$.
By considering the line search stopping condition \eqref{eq:armijo}
and using $\lambda = \mu_B / 2$ in \eqref{eq:bound_improve},
we can see that since $B_t$ is positive semidefinite,
\begin{align*}
f\left( \AL^{t} + \eta_t \Delta\AL^t\right) - f\left( \AL^t \right)
&\leq \eta_t \tau \Delta_t
\leq \eta_t \tau \left(\Delta_t + \frac12 \left\| \Delta\AL^t
\right\|_{B_t}^2 \right)
= \eta_t \tau Q^{\AL^t}_{B_t} \left( \DAL^t \right)\\
&\leq \eta_t \tau \left( 1 - \gamma \right)Q^{\AL^t}_{B_t} \left(
\left( \Delta\AL^t \right)^*\right)
\leq -\eta_t \tau \left( 1 - \gamma \right) \frac{\mu_B}{4} \left(
f\left( \AL^t \right) - f^* \right).
\end{align*}
Finally, by taking in the step size in Lemma~\ref{lemma:improvedline},
we see that the convergence speed is
\begin{align*}
f\left( \AL^{t+1} \right) - f^* &= f\left( \AL^{t} + \eta_t
\Delta\AL^t\right) - f^*\\
&\leq \left( 1 - \frac{\tau \left( 1 - \gamma \right) \mu_B}{4}
	\min\left\{ 1, \frac{2 \beta \left( 1 - \tau
	\right)}{\left( 1 + \sqrt{\gamma} \right) L_B} \right\} \right)
	\left( f\left( \AL^t \right) - f^* \right)\\
	&= \left(1 - \frac{\tau \mu_B}{2}\min\left\{\frac{1 - \gamma}{2}, \frac{\beta
		\left( 1 - \tau \right) \left( 1 - \sqrt{\gamma}
	\right)}{L_B}\right\}\right) \left( f\left( \AL^t \right) - f^* \right).
\end{align*}
\end{proof}
\end{document}